\documentclass[journal,onecolumn, 12pt]{IEEEtran}

\usepackage[margin=1in]{geometry}

\usepackage[utf8]{inputenc} % allow utf-8 input
\usepackage[T1]{fontenc}    % use 8-bit T1 fonts
\usepackage[colorlinks=true,
            linkcolor=blue,
            citecolor=red,
            urlcolor=blue,
            anchorcolor=red,
            bookmarks=true,
            plainpages=false]{hyperref}
            
\usepackage{url}            % simple URL typesetting
\usepackage{nicefrac}       % compact symbols for 1/2, etc.
\usepackage{microtype}      % microtypography
\usepackage{xcolor}         % colors
\usepackage{comment}

\usepackage[cmex10]{amsmath}
\usepackage{amssymb}
\usepackage{amsfonts}
\usepackage{amsbsy}
\usepackage{amsthm}
\usepackage{thmtools}
\usepackage{thm-restate}
\usepackage{mathtools}
\usepackage{bbm}
\usepackage[capitalize,nameinlink]{cleveref}
\usepackage{xfrac}

\usepackage{subfigure}
\usepackage{float}
\usepackage{dsfont}
\usepackage{graphicx}
\usepackage{booktabs}
\usepackage{color}
\usepackage{url}
\usepackage{cite}
\usepackage{tabularx}

\usepackage{pdfpages}
\usepackage{enumitem}

\usepackage[linesnumbered,ruled]{algorithm2e}

\newtheorem{theorem}{Theorem}[section]%[section]
\newtheorem{corollary}{Corollary}[section]
\newtheorem{lemma}{Lemma}[section]
\newtheorem{proposition}[theorem]{Proposition}

\newtheorem{remark}{Remark}[section]

\newcounter{thmcase}[theorem]

\crefname{thmcase}{theorem}{theorems}
\Crefname{thmcase}{Theorem}{Theorems}

\newcommand{\thmcaseLabel}[1]{%
   \refstepcounter{thmcase}% 
   \label{#1}% 
}

\usepackage{parskip}
\newcommand{\eps}{\varepsilon}
\newcommand{\dd}{\mathrm{d}}
\newcommand{\ci}{\perp\!\!\!\perp}

\renewcommand{\P}{\mathsf{P}}
\newcommand{\Q}{\mathsf{Q}}
\newcommand{\Expect}{\mathbb{E}}

\newcommand{\TV}{\mathrm{TV}}
\newcommand{\ind}[1]{\mathbf{1}_{\{#1\}}}

\DeclareMathOperator{\Bern}{Bern}

\DeclareMathOperator{\Bin}{Bin}
\DeclareMathOperator{\ER}{ER}
\DeclareMathOperator{\Var}{Var}
\DeclareMathOperator{\Cov}{Cov}

\DeclareMathOperator{\Id}{Id}

\DeclareMathOperator{\ov}{ov}

\newcommand{\calE}{\mathcal{E}}
\newcommand{\En}{\mathcal{E}_n}

\newcommand{\Sn}{\mathfrak{S}_n}

\usepackage[size=tiny]{todonotes}

\newcommand{\Indc}{\mathbf{1}}

\def\ci{\perp\!\!\!\perp}

\renewcommand{\Pr}[1]{\mathsf{P}\left(#1 \right)} % probability
\newcommand{\Qr}[1]{\mathsf{Q}\left(#1 \right)}

\newcommand{\erdosrenyi}{Erd\H{o}s-R\'{e}nyi}

\newcommand{\pistar}{\pi^{\star}}

\renewcommand{\hat}{\widehat}

\newcommand{\G}{\mathsf{G}}
\newcommand{\E}{\mathsf{E}}

\usepackage{etoolbox} 



\begin{document}
\title{Converse bounds for multiple graph alignment and correlation detection based on last matching} 

\author{
    Taha Ameen %\orcidlink{0000-0001-5449-0840},~\IEEEmembership{Student Member,~IEEE,} 
    and Bruce Hajek %\orcidlink{0000-0002-8520-0196},~\IEEEmembership{Life Fellow,~IEEE}  
        \thanks{
            The authors are with the Department of Electrical and Computer Engineering and the Coordinated Science Laboratory, University of Illinois, Urbana, IL 61801, USA. E-mails: \{ \texttt{tahaa3, b-hajek} \} \texttt{@illinois.edu}. A portion of this paper, for alignment in Gaussian model, appeared in the proceedings of IEEE ISIT 2025 \cite{ameen2025detecting}.
        }
    }

\date{}

\maketitle

\begin{abstract}
The paper focuses on information theoretic converse bounds for the alignment of $m$ correlated graphs and for the detection of correlation among $m$ graphs.   A simple idea for $m\geq 3$ is that if the alignment of $m-1$ of the graphs is revealed as extra information (by a genie for example) then it is still necessary to produce the alignment between the one remaining graph and the others, i.e. the last matching must be accomplished. For both Gaussian and \erdosrenyi\ models the last-matching problem is equivalent to one with two observed graphs, providing a path to extend converse bounds for $m=2$ to larger $m.$   While the method is rather obvious for alignment, we show that the method can also be used to derive converse bounds for weak detection of correlation.
\end{abstract}

%\begin{IEEEkeywords}
%    Hypothesis testing, circular statistics, von Mises distribution, planted models, community detection
%\end{IEEEkeywords}

\tableofcontents

\section{Introduction}
\label{sec:introduction}

Networks describing the same collection of objects are increasingly observed across different platforms, modalities, and time periods.  Such networks are often statistically related, but their vertex labels need not agree.  For example, users may be anonymized differently across social networks, proteins may be indexed differently across species, and brain connectomes obtained from different subjects have no natural common labeling.  Recovering the latent correspondence can enable linkage across social networks~\cite{narayanan2008robust,narayanan2009deanonymizing}, identify conserved biological structure~\cite{singh2008global}, and facilitate comparisons between brain networks~\cite{sporns2005human,calissano2024graph}.

Two fundamental statistical problems arise in this setting.  In \emph{graph alignment}, the graphs are assumed to be correlated and the objective is to recover their latent vertex correspondence.  In \emph{correlation detection}, the more basic objective is to determine whether the observed graphs are correlated through some unknown alignment or are instead mutually independent.  Both problems become particularly interesting when more than two graphs are observed.  The additional graphs provide more evidence about the common latent structure, and can make alignment or detection possible even when every individual pair of graphs contains too little information.

Determining the benefit of these additional observations requires both achievability results and converse bounds.  Direct converse arguments for $m$ graphs can be substantially more difficult than their two-graph counterparts: the latent alignment now consists of $m-1$ permutations, and the likelihood ratio or posterior distribution involves their joint overlap structure.  At the same time, a considerable literature has already developed powerful converse techniques for two correlated graphs.  This raises a natural question: \emph{Can existing two-graph converses be converted into converse bounds for an arbitrary fixed number of graphs?}

This paper develops a simple method for doing so.  Imagine that a genie reveals the information needed to align graphs $2,\ldots,m$ with one another.  Even with this additional information, graph $1$ must still be matched to the aligned collection of the remaining graphs.  Moreover, in the Gaussian and \erdosrenyi\ models, the aligned collection can be compressed without loss of relevant information into a single aggregate graph: a sum in the Gaussian model and an edgewise union in the \erdosrenyi\ model.  The remaining task is therefore a two-graph problem between graph $1$ and the aggregate graph.  We call this the \emph{last-matching principle}.

For alignment, the validity of this principle is immediate at an intuitive level: providing the internal alignment of $m-1$ graphs can only make estimation easier, so impossibility after this revelation implies impossibility without it.  Its use for correlation detection is less direct.  Under the null hypothesis there is no true alignment for the genie to reveal, and the information supplied by the genie must not itself disclose which hypothesis is true.  We resolve this issue by constructing a plausible genie output under the null and proving a recursive total-variation bound.  This allows two-graph detection converses to be lifted inductively to any fixed number of graphs.

In the \erdosrenyi\ model, the resulting two-graph problem is asymmetric (see \Cref{sec:problem-formulation} for a precise definition). However, most existing two-graph converses are stated for the symmetric case in which the two retention probabilities coincide.  A second purpose of this paper is therefore to establish the asymmetric versions of the weak-detection and alignment converses needed by the last-matching reduction.

\subsection{Related work}
\label{sec:related_work}

\paragraph{Two correlated graphs.}
The subsampling model for correlated \erdosrenyi\ graphs has become a standard framework for the theoretical study of graph alignment~\cite{pedarsani2011privacy}.  Exact-alignment thresholds were studied by Cullina and Kiyavash~\cite{cullina2016improved,cullina2017exact}, while almost-exact and partial alignment were subsequently investigated in
\cite{cullina2019kcore,wu2022settling,ding2023densesubgraph,hall2023partial,du2025optimal}.
For Gaussian graphs, sharp alignment thresholds were obtained independently by Ganassali~\cite{ganassali2022sharp} and by Wu, Xu, and Yu~\cite{wu2022settling}.

The hypothesis-testing problem for two unlabeled graphs was studied by Wu, Xu, and Yu~\cite{wu2023testing}, who obtained sharp thresholds for Gaussian graphs and sufficiently dense \erdosrenyi\ graphs, together with converse bounds in the sparse setting.  The sharp threshold in the moderately sparse \erdosrenyi\ regime was subsequently determined by Ding and Du~\cite{ding2023detection}, while the constant-average-degree regime was studied by Feng~\cite{feng2025strong}.  These works supply the two-graph converses on which the present reduction builds.

\paragraph{Multiple correlated graphs.}
The use of several graphs for alignment has been investigated for \erdosrenyi\ graphs and more general random-graph models in
\cite{josephs2021recovery,ameen2024exact,ameen2024aligning}, and for multiple correlated stochastic block models in~\cite{racz2024harnessing}.  In the Gaussian setting, Vassaux and Massouli\'e~\cite{vassaux2025} determined the alignment threshold for a fixed number of graphs.  Even and Ganassali~\cite{even2025statistical} considered the regime in which the number of graphs $m$ can grow, possibly quickly, with $n$, without being concerned with constant factors for the thresholds.  The last-matching approach was used to establish a converse bound for weak detection in the Gaussian setting in~\cite{ameen2025detecting}. Sharp Gaussian detection thresholds were obtained in~\cite{ameenhajekGaussian}, and sparse multi-graph \erdosrenyi\ detection was recently studied in~\cite{ochoa2026detection}. 

These works primarily analyze the multi-graph problem directly.  The present paper takes a complementary approach: it asks what can be obtained from existing two-graph converses through a general last-matching reduction.  The resulting bounds need not always be sharp, since revealing the alignment of $m-1$ graphs may give substantial additional information.  Nevertheless, the method is modular, applies simultaneously to alignment and detection, and identifies the last matching as a necessary obstruction to solving the full problem.

\subsection{Contributions}
\label{sec:contributions}

The main contributions of the paper are as follows.

\begin{enumerate}[leftmargin=*]

\item \textbf{A last-matching converse for alignment.}
We show that after a genie reveals the alignment among graphs $2,\cdots,m$, their relevant information can be aggregated into a single graph $Y$.  Consequently, any impossibility result for this reduced two-graph problem between the first graph and $Y$ yields an impossibility result for the original $m$-graph problem.  The reduction applies to exact, almost-exact, and partial alignment.

\item \textbf{A last-matching converse for detection.}
We develop an analogous reduction for weak detection.  The principal difficulty is defining the genie's output under the null hypothesis, where no planted alignment exists.  We construct a suitable null output and prove that the total variation distance for the $m$-graph problem is bounded by the total variation distance of the last two-graph problem, plus the corresponding distance for the remaining $m-1$ graphs.  Iterating this inequality transfers two-graph weak-detection converses to the multi-graph setting.

\item \textbf{Asymmetric two-graph \erdosrenyi\ converses.}
We establish the asymmetric versions of existing weak-detection and partial-alignment converses, with subsampling probabilities $s_1$ and $s_2$ replacing the common retention probability $s$.  The proofs adapt the conditional second moment arguments of
\cite{wu2023testing,feng2025strong}, the densest-subgraph arguments of
\cite{ding2023detection,ding2023densesubgraph}, and the mutual-information MMSE area method of~\cite{wu2022settling}.  We also record existing asymmetric converses for almost-exact and exact alignment.

\item \textbf{Applications to Gaussian and \erdosrenyi\ graphs.} We apply the last-matching converse to obtain converse bounds for exact alignment, almost-exact alignment, partial alignment and weak detection, respectively in 
the \erdosrenyi\ model and the Gaussian model.
\end{enumerate}

The last-matching converse bound was proved to be tight for exact matching in the particular sparse scaling regime $p=(c\log n)/n$ \cite{ameen2024exact,racz2024harnessing}.
However, comparison of the last-matching converse bounds with the sharp thresholds in the Gaussian case in~\cite{vassaux2025,ameenhajekGaussian} indicates the last-matching converse bound is not always the best possible: the last-matching converse is a broadly applicable black-box reduction, but can be strictly weaker than a direct analysis of the full multi-graph problem.

\paragraph{Organization.}
\Cref{sec:problem-formulation} introduces the Gaussian and \erdosrenyi\ models and defines the alignment and detection objectives.  It also includes a compact listing of our multi-graph converse bounds.  \Cref{sec:genie-recovery} develops the last-matching reduction for alignment and applies it to the Gaussian model, while \Cref{sec:genie-detection} establishes the corresponding last-matching reduction for weak detection.  Because the \erdosrenyi\ reduction produces a two-graph problem with unequal retention probabilities, \Cref{sec:converses_asymmetric} states the required asymmetric two-graph converses for weak detection and for exact, almost-exact, and partial alignment. \Cref{sec:ER_corollaries} combines these converses with the last-matching reductions to obtain the final converse bounds for multiple correlated \erdosrenyi\ graphs.  The proofs of the asymmetric detection and partial-alignment converses are given in Appendices~\ref{app:detection_converses_asymmetric} and~\ref{app:converses_asym_recovery} respectively, and \Cref{sec-discussion} concludes with a discussion of the scope and limitations of the last-matching method and it also includes a new positive result for detection of correlation for multiple \erdosrenyi\ graphs in the dense regime.

\section{Problem Formulation} \label{sec:problem-formulation}

Fix an integer $m\ge 2$ independent of $n$, and for any finite set $A$, denote by $\binom{A}{2}\coloneqq\{B\subset A: |B|=2\}$ the set of unordered pairs from $A$.  Let 
\[
        \En = \binom{[n]}{2},
        ~~~~~~
        N = |\En| = \binom{n}{2} .
\]
We observe $X$ consisting of $m$ random symmetric $n\times n$ matrices $X=(X^1, X^2, \cdots, X^m)$ with zero diagonals corresponding to (possibly weighted) graphs on the common vertex set $[n]$, where $X^k = (X^k_{ij})_{i,j\in [n]}.$
We are interested in both alignment (determining the latent correspondence) and detection (determining whether correlation is present)
from observation of $X$.

Let $\Sn$ denote the set of permutations of $[n].$  A permutation $\tau\in\Sn$ induces a permutation on $\En$, also denoted by $\tau$, such that for $e=\{i,j\} \in \En$ with $i\neq j$, $\tau(e) = \{\tau(i),\tau(j)\}.$
An {\em alignment} for $m$ matrices is an $m$-tuple of permutations
\[
        \pi=(\pi_1,\cdots,\pi_m),
        ~~~~ \mbox{ where }
        \pi_1=\Id.
\]
Given an alignment $\pi$,  related permutations $(\pi_{k\ell})_{k\ell\in [m]}$ are defined by 
$\pi_{k \ell} = \pi_{\ell}\circ (\pi_{k})^{-1}$ for $k,\ell \in [m].$
Note that, since $\pi_1 = \Id,$ $\pi_{1\ell} = \pi_{\ell}$ for $\ell\in [m].$

Suppose that under a probability distribution $\P$ there exists a uniformly distributed random alignment $\pistar$ such that, conditioned on $\pistar$,
\begin{align}
    \big( X^1_e, X^2_{\pistar_{2}(e)},\cdots, X^m_{\pistar_{m}(e)} \big)_{e\in \En}
\label{eq:aligned_vectors}
\end{align}
are $N$ mutually independent $m$-vectors.   Consider the following two cases for the distribution of the $m$-vectors.

{\bf Gaussian case:} Let $\rho \in [0,1].$ Under $\P$ the $m$-vectors consist of $m$ jointly Gaussian variables with mean zero, variance one and pairwise covariance $\rho.$

{\bf \erdosrenyi\ case:} Let $s,p\in [0,1].$ 
For each vertex pair $e\in\En$,
\[
        X^k_{\pistar_k(e)}
        =
        X^0_e \, \xi^{k}_e,
        ~~~~~~
        X^0_e\sim \Bern(p),
        ~~~~~~
        \xi^{k}_e\sim \Bern(s),
\]
with all $(X^0_e)$ and $(\xi^{k}_e)$ mutually independent.  Note that $X^0$ is the incidence matrix for an \erdosrenyi$(n,p)$ graph and for each $k$, $X^k$ is an \erdosrenyi$(n,ps)$ graph obtained by subsampling the parent graph $X^0$ with retention probability $s.$

\paragraph{Alignment}

Following \cite{vassaux2025}, let the overlap between
two alignments $\pi$ and $\pi'$, denoted $\ov(\pi,\pi')$, be the fraction of vertices that the two alignments agree upon over all $m$ graphs:
\begin{equation} \label{strongov}
   \ov(\pi, \pi') = \frac{1}{n} \sum_{i\in [n]} \ind{\pi_2(i) = \pi'_2(i), \cdots, \pi_m(i) = \pi'_m(i)}.
\end{equation}
Let $\hat{\pi}$ be an estimator of $\pistar$ based on observation of $X.$
The estimator achieves
 \begin{itemize}
\item \textit{Exact alignment} if the probability of any errors converges to $0$:
    \[ 
     \Pr{\ov(\pistar,\hat{\pi}) = 1}  = 1 - o(1).
    \]
\item \textit{Almost-exact alignment} if the fraction of correctly matched vertices converges to one: 
    \[ 
        \Pr{\ov(\pistar,\hat{\pi}) \geq 1 -\epsilon}  = 1 - o(1)~~~\mbox{for all } \epsilon > 0.
    \]
\item \textit{Partial alignment} if the fraction of correctly matched vertices exceeds some positive constant with probability converging to one: 
    \[ 
     \Pr{\ov(\pistar,\hat{\pi}) \geq c}  = 1 - o(1)~~~\mbox{for some } c > 0.
    \] 
\end{itemize}
Given the alignment problem,
we say exact alignment is impossible if there does not exist an estimator achieving exact alignment,  almost-exact alignment is impossible if there does not exist an estimator achieving almost-exact alignment,  and partial alignment is impossible if there does not exist an estimator achieving partial alignment.
The following is a stronger negative property than impossibility of partial alignment.
 \begin{itemize}
\item \textit{Intractability of partial alignment.} Partial alignment is said to be intractable if for every estimator $\hat{\pi},$
    \[ 
     \Pr{\ov(\pistar,\hat{\pi}) \leq \epsilon}  = 1 - o(1)~~~\mbox{for all } \epsilon > 0.
    \]
\end{itemize}

\paragraph{Detection}

We consider binary hypothesis testing problems such that under hypothesis $H_1$ the observation $X$ has distribution $\P$ as described above.   Under the null hypothesis the observation $X$ has distribution $\Q$ given by
$\Q_X = \P_{X^1} \otimes \P_{X^2} \otimes \cdots \otimes \P_{X^m}.$  In other words, under $\Q$ the matrices $X^1, \cdots , X^m$ are mutually independent with the same marginal distributions as under $\P.$
A test statistic $T(X)$ with threshold $\tau$ achieves 
\begin{itemize}
    \item \textit{Strong detection} if the total error converges to $0$:
    \[ 
        \Pr{T(X) < \tau} + \Qr{T(X) \geq \tau} = o(1).
    \]
    \item \textit{Weak detection} if the test outperforms random guessing:
    \[ 
        \Pr{T(X) < \tau} + \Qr{T(X) \geq \tau} = 1 -\Omega(1).
    \]
\end{itemize}
Given the hypothesis testing problem, we say strong detection is impossible if there does not exist a decision rule achieving strong detection and weak detection is impossible if there does not exist a decision rule achieving weak detection.

Let $\TV(\P,\Q)$ denote the total variation distance between the two measures $\P$ and $\Q$.   (More precisely, it denotes $\TV(\P_X,\Q)$, the $\TV$ distance between the marginal distribution of $X$ under $\P$ and $\Q.$)   Impossibility of weak detection is equivalent to $\TV(\P,\Q) = o(1).$

\subsection{Summary of converse bounds from last-matching method for multiple \erdosrenyi\ graphs}

\begin{table*}[t]
\centering
\caption{
Converse bounds for weak detection and partial alignment from \Cref{cor:asym_er_weak_detection_converses,cor:asym_dd_recovery}.
%Overview of converse bounds obtained from the last-matching reduction.  
Here \(\epsilon>0\) is an arbitrarily small fixed constant and \(\lambda_\alpha=\varrho^{-1}(1/\alpha)\), where \(\varrho\) is defined in~\eqref{eq:def_varrho}.  
\medskip
}
\label{tab:main_converse_results}

\small
\renewcommand{\arraystretch}{1.45}

\begin{tabularx}{\textwidth}{
    @{}
    >{\raggedright\arraybackslash}p{0.26 \textwidth}
    >{\arraybackslash}X
    >{\arraybackslash}X
    %     >{\centering\arraybackslash}X
    %     >{\centering\arraybackslash}X
    @{}
}
\toprule
\textbf{Model and regime}
&
\textbf{Weak detection is impossible}
&
\textbf{Partial alignment is intractable}
\\
\midrule

%Gaussian
%&
%\(\displaystyle
%    \rho^2
%    \leq
%    \frac{4-\epsilon}{m-1}\frac{\log n}{n}
%\)
%&
%\(\displaystyle
%    \rho^2
%    \leq
%    \frac{4-\epsilon}{m-1}\frac{\log n}{n}
%\)
%\\
%
%\addlinespace

\erdosrenyi, sparse
\newline
{\footnotesize (includes $p=n^{-1+o(1)}$) }   
&
\(\displaystyle
    nps^2
    \leq
    \frac{1-\omega(n^{-1/3})}{m-1} {\footnotesize \mbox{ (and } s\to 0)}
\)
&
\(\displaystyle
    nps^2
    \leq
    \frac{1-\epsilon}{m-1}{\footnotesize \mbox{ (and } np = \omega(\log^2 n))}
\)
\\

\addlinespace

\erdosrenyi, moderately sparse
\newline
{\footnotesize \(p=n^{-\alpha+o(1)}\), \(0<\alpha<1\)}
&
\(\displaystyle
    nps^2
    \leq
    \frac{\lambda_\alpha-\epsilon}{m-1}
\)
&
\(\displaystyle
    nps^2
    \leq
    \frac{\lambda_\alpha-\epsilon}{m-1} 
\)
\\

\addlinespace

\erdosrenyi, dense
\newline
{\footnotesize \(p=n^{-o(1)}\) (and $p\leq 1-\Omega(1)$)}
&
\(\displaystyle
\begin{gathered}
    nps^2\leq
     \frac{(2-\epsilon)\log n}{(m-1)(\log\frac 1 p -1 + p)}
\end{gathered}
\)
&
\(\displaystyle
\begin{gathered}
        nps^2\leq
     \frac{(2-\epsilon)\log n}{(m-1)(\log\frac 1 p -1 + p)} \\
\end{gathered}
\)
\\
\bottomrule
\end{tabularx}
\end{table*}

The converse bounds from \Cref{cor:asym_er_weak_detection_converses,cor:asym_dd_recovery} are listed in
\Cref{tab:main_converse_results}.
Additionally, \Cref{cor:almost_exact_recovery} shows that almost-exact alignment is impossible if
\(
        nps^2=O(1)
\)
and \Cref{cor:asym_exact_recovery} shows that exact alignment is impossible if $ps=o(1)$ and
\(
        nps^2(1-\sqrt p)^2
        \leq
        \frac{(1-\epsilon)\log n}{m-1}.
\)

\section{Last-matching converse method for alignment} \label{sec:genie-recovery}

In the following we use the notation $X^{2:m} = (X^2,\cdots ,X^m).$
Let $\pistar_{2:m} = (\pistar_{k\ell}:2\leq k,\ell \leq m),$ which determines how $X^2, \cdots , X^m$ are aligned under $\pistar.$ 
Consider the performance of estimators of $\pistar$ that are based on observation of $(X,\pistar_{2:m}).$   Since such estimators could ignore the additional information $\pistar_{2:m}$, if any  type of alignment is impossible with this additional information it is also impossible for the original alignment problem.

Define the overlap for two permutations $\tau,\tau'\in \Sn$ by $\ov(\tau,\tau')=\frac 1 n |\{i\in [n]: \tau(i)=\tau'(i)\}|.$  In particular, if $\pistar$ is an alignment of $X=(X_1, \ldots , X_m)$ and $\hat{\pi}$ is an estimator of $\pistar,$  then $\ov(\pistar_2,\hat{\pi}_2)$ is the fraction of vertices correctly aligned between $X_1$ and $X_2.$   Note that $\ov(\pistar,\hat{\pi}) \leq\ov(\pistar_2,\hat{\pi}_2)$
with equality if and only if for all $i\in [n]$,
\begin{align}
    \widehat{\pi}_2(i) = \pistar_2(i) ~~~  \iff   ~~~ \widehat{\pi}_k(i) = \pistar_k(i) \mbox{ for } 2\leq k \leq m.   \label{eq:all_or_nothing}
\end{align}
Given any estimator $\hat{\pi}$ based on observation of $(X,\pistar_{2:m}),$  the new estimator $\hat{\pi}'$ defined by $\hat{\pi}'_k = \pistar_{2k}\circ \hat{\pi}_2$ for $2\leq k\leq m$ satisfies  \eqref{eq:all_or_nothing} and also $\hat{\pi}'_2=\hat{\pi}_2.$  
Therefore, 
$\ov(\pistar,   \hat{\pi} ) \leq 
\ov(\pistar_2, \hat{\pi}_2) = 
\ov(\pistar_2, \hat{\pi}'_2)  =
\ov(\pistar, \hat{\pi}').$
Hence, for the alignment problem with observation $(X,\pistar_{2:m}),$  we can consider without loss of optimality estimators of the form $\hat{\pi}'$.   Therefore, we can change the objective from maximizing $\ov(\pistar,\hat{\pi})$ over all choices of alignment estimator $\hat{\pi}$ to maximizing $\ov(\pistar_2,\hat{\pi}_2)$ over all choices of a single permutation estimator $\hat{\pi}_2$ of $\pistar_2$ (still based on observation of $(X,\pistar_{[2:m]})$).

Define an $n\times n$ symmetric matrix $Y$ as a function of $(X^{2:m},\pistar_{2:m})$ as follows.   For the Gaussian case,  $Y_e = X^2_e + X^3_{\pistar_{23}(e)} + \cdots  + X^m_{\pistar_{2m}(e)}$ which is the sum of $X^2$ through $X^m$ after alignment by $\pistar.$ 
Similarly, for the \erdosrenyi\ case, $Y_e = X^2_e\vee X^3_{\pistar_{23}(e)}\vee \cdots  \vee X^m_{\pistar_{2m}(e)}  $, which is the pointwise maximum of $X^2$ through $X^m$ after alignment by $\pistar.$ 

\medskip

\begin{lemma}  \label{lemma:basic_Markov}
$(X^1,\pi_2^*) - Y - (X^{2:m},\pistar_{2:m})$ forms a Markov chain under $\P.$
\end{lemma}

\begin{proof}   It is sufficient to show that the conditional distribution of $(X^1,\pi_2^*)$
given $(X^{2:m},\pistar_{2:m})$ depends only on $Y.$   First, note that $\pistar_2$ is independent of $(X^{2:m},\pistar_{2:m})$ and uniformly distributed over $\Sn.$ 
Thus, to complete the proof of the lemma, it suffices to show that the conditional distribution of $X^1$ given $(X^{2:m},\pistar_{2:m},\pistar_2)$ depends only on $(Y,\pistar_2).$  Equivalently, since $\pistar_{2:m}$ and $\pistar_2$ together completely determine the alignment $\pistar$,  it suffices to show that the conditional distribution of $X^1$ given $(X^{2:m},\pistar)$ depends only on $(Y,\pistar_2).$ 

In view of the independence of the vectors in \eqref{eq:aligned_vectors}, it follows that given $(X^{2:m},\pistar)$ the entries $(X^1_e)_{e\in\En}$ are conditionally independent.  We thus need only consider for fixed $e$ the conditional distribution of $X^1_e$ given
$(X^{2:m},\pistar),$ which is the same as the conditional distribution of $X^1_e$ given
$(X^2_{\pistar_2(e)}, X^3_{\pistar_3(e)} \cdots , X^m_{\pistar_m(e)}),$ and show that such
conditional distribution is determined by $(Y,\pistar_2).$  This we do separately for the Gaussian case and \erdosrenyi\ case.

For the Gaussian case, given $\pistar$, the $m$-vector $(X^1_e, X^2_{\pistar_2(e)}, X^3_{\pistar_3(e)} \cdots , X^m_{\pistar_m(e)})$ has the ${\cal N}(0, (1-\rho)I + \rho J)$
distribution and $Y_{\pistar_2(e)} = X^2_{\pistar_2(e)} + \cdots + X^m_{\pistar_m(e)}.$  So the conditional distribution of $X^1_e$ given the other $m-1$ variables is Gaussian with conditional mean
\begin{align*}
\frac{\Cov (X^1_e,Y_{\pistar_2(e)})Y_{\pistar_2(e)}}{\Var(Y_{\pistar_2(e)})} = \frac{(m-1)\rho Y_{\pistar_2(e)}}{(m-1) + (m-1)(m-2)\rho}
= \frac{\rho Y_{\pistar_2(e)}}{1 + (m-2)\rho}
\end{align*}
and the variance of the conditional distribution depends only on $\rho$ and $m.$   Thus, the conditional distribution of $X^1_e$ given $(X^2_{\pistar_2(e)}, X^3_{\pistar_3(e)} \cdots , X^m_{\pistar_m(e)})$ only depends on $Y_{\pistar_2(e)}$, which is determined by $(Y,\pistar_2),$ completing the proof of the lemma in the Gaussian case.

Similarly, for the \erdosrenyi\ case, given $\pistar$, the $m$-vector $(X^1_e, X^2_{\pistar_2(e)}, X^3_{\pistar_3(e)} \cdots , X^m_{\pistar_m(e)})$ is given by
$(X^0_e\xi^1_e, \cdots , X^0_e\xi^m_e).$    Thus, if  $B = \big(X^2_{\pistar_2(e)}, X^3_{\pistar_3(e)} \cdots , X^m_{\pistar_m(e)} \big),$ then for any $b\in \{0,1\}^{m-1},$
\begin{align*}
    \P (X^1_e = 1 | B=b ) & = \frac{\P \big(\{X^1_e=1\}\cap \{B=b\}\big)}{P(B=b)} 
    = 
    \begin{cases} 
        s \,,  & \mbox{if } |b| \neq 0 \\
        \displaystyle\frac{ps(1-s)^{m-1}}{p(1-s)^{m-1} + 1 - p} \,, & \mbox{if } |b|=0 
    \end{cases} \, .
\end{align*}
Therefore, $\P (X^1_e = 1 | B )$ is determined by $\Indc_{\{|B|\neq 0\}}$ which is equal to $Y_{\pistar_2(e)}.$  Thus, the conditional distribution of $X^1_e$ given $(X^2_{\pistar_2(e)}, X^3_{\pistar_3(e)} \cdots , X^m_{\pistar_m(e)})$ only depends on $Y_{\pistar_2(e)}$, which is determined by $(Y,\pistar_2),$ completing the proof of the lemma in the \erdosrenyi\ case.
\end{proof}

\Cref{lemma:basic_Markov} implies that $\pistar_2 - (X^1,Y)-(X^{2:m},\pistar_{2:m})$ is also a Markov chain under $\P.$
So for the purpose of recovering $\pistar_2$ from observation of $(X,\pistar_{2:m}),$ it suffices to consider estimators $\widehat{\pi}_2$ that depend only on $(X^1, Y)$ for achieving any of the alignment properties defined above.
We therefore have proved the following theorem.

\medskip

\begin{theorem}  \label{thm:recovery_genie}
\begin{enumerate}[label=$\mathrm{(\alph*)}$]
    \item If exact (resp. almost-exact, partial) recovery of $\pistar_2$ from $(X^1,Y)$ is impossible then exact (resp. almost-exact, partial) recovery of $\pistar$ from $X$ is impossible. 
    \item If partial recovery of $\pistar_2$ from $(X^1,Y)$ is intractable then partial recovery of $\pistar$ from $X$ is intractable.
\end{enumerate}
\end{theorem}

Note that exact (resp. almost-exact, partial) recovery of $\pi_2^*$ from $(X^1,Y)$ is equivalent to exact (resp. almost-exact, partial) alignment of $(X^1,Y)$; similarly, recovery of $\pi^*$ corresponds to alignment of $X$.
In order to apply \Cref{thm:recovery_genie} we need to identify the joint distribution of $\pistar_2, X^1, Y$ in both the Gaussian and \erdosrenyi\ cases.  In either case, $\pistar_{2}$ is uniformly distributed over $\Sn$ and independent of $X^1$ and, given $\pistar_{2}$,  the pairs $(X^1_e, Y_{\pistar_2(e)})_{e\in \En}$ are mutually independent.

\subsection{Application to the Gaussian model}
\label{sec:Gaussian_case_recovery}
For the Gaussian case, given $\pistar_{2}$ and for any $e\in \En,$  $X^1_e$ and  $Y_{\pistar_2(e)}$ are jointly Gaussian with mean zero.  $Y_{\pistar_2(e)}$ is the sum of $m-1$ standard normal variables with correlation $\rho$  among themselves and with $X^1_e.$  Therefore,
\begin{align*}
    \Var(Y_{\pistar_2(e)}) = (m-1) + (m-1)(m-2) \rho ~~~~ \Cov(X^1_e,Y_{\pistar_2(e)})=(m-1)\rho.
\end{align*}
Hence, for any $e\in \En,$ the correlation coefficient between $X^1_e$ and $Y_{\pistar_2(e)}$ given $\pistar_2$ is
\begin{align*}
  \rho_m =   \frac{(m-1)\rho}{\sqrt{1}\sqrt{(m-1) + (m-1)(m-2)\rho} } = \rho \times \sqrt{\frac{m-1}{(m-2)\rho + 1}}.
\end{align*}
Recovering $\pistar_2$ from $(X^1,Y)$ is thus equivalent to the Gaussian alignment problem for two matrices with $\rho$ replaced by $\rho_m.$    By~\cite{wu2022settling} partial alignment is impossible if $n\rho_m^2 \leq (4-\epsilon')\log n$ or equivalently:
\begin{align}   \label{eq:n_rho2}
    n \rho^2 \times \frac{m-1}{(m-2)\rho + 1} \leq (4-\epsilon') \log n,
\end{align}
for some fixed $\epsilon'>0.$   (Note that even though the distribution of $(X^1,Y)$ is not symmetric in $X^1$ and $Y$, the equivalent observation $(X^1,Y/\sqrt{(m-1) + (m-1)(m-2)\rho})$ is symmetric, and only the correlation coefficient $\rho_m$ matters.)
The factor $(m-1)/((m-2)\rho +1)$ equals $1$ in the well-studied setting of $m=2$, and captures the role of the number of observations. As $m\to\infty$, the factor converges to $1/\rho$.

As $n\to\infty$, \eqref{eq:n_rho2} requires $\rho\to 0$ so we can absorb the effect of the denominator in the LHS into the $\epsilon'$ on the RHS. So the above condition is equivalent to:
\[
     \rho^2 \leq \frac{4-\epsilon}{m-1} \cdot \frac{\log n}{n}
\]
for some $\epsilon>0$. This yields the following corollary of the converse in
\cite{wu2022settling} and \Cref{thm:recovery_genie}.

\medskip 

\begin{corollary} \label{cor:Gaussian_recovery_converse}
(Gaussian model, $m\geq 2$)
    Let $\epsilon >0$ denote an arbitrarily small but fixed constant. If
    \[
        \rho^2 \leq \frac{4-\epsilon}{m-1} \cdot \frac{\log n}{n},
    \]
    then partial alignment is intractable. (Hence exact alignment and almost-exact alignment are impossible.)
\end{corollary}
This corollary is not tight.  It is shown in \cite{vassaux2025} that partial alignment is intractable if $\rho^2 \leq \frac {8-\epsilon} m \cdot \frac{\log n} n.$  That is the sharp threshold -- \cite{vassaux2025} also shows that exact alignment is possible if $\rho^2 \geq \frac {8+\epsilon} m \cdot \frac{\log n} n.$     Roughly speaking, we can conclude that in the Gaussian case, it is strictly more difficult to align all $m$ graphs than it is to align the last graph if the alignment of the other $m-1$ graphs is given.

\subsection{Towards application to the \erdosrenyi\ model}
\label{sec:towards_ER}

For the \erdosrenyi\ case, given $\pistar_{12}$ and for any $e\in \En,$  $X^1_e$ and 
$Y_{\pistar_2(e)}$ are binary random variables.

We need to identify the joint distribution of $X^1_e, Y_{\pistar_2(e)}$ for each $e$.
Note that
\begin{align*}
X^1_e = X^0_e \, \xi^1_e \hspace{1.5cm}
    Y_{\pistar_2(e)}  = X^2_{\pistar_2(e)}\vee X^3_{\pistar_3(e)}\vee \cdots  \vee X^m_{\pistar_m(e)} =  X^0_e \, \widetilde{\xi^2_e}
\end{align*}
where
\begin{align*}
    \widetilde{\xi^2_e} =   \xi^2_e \vee  \xi^3_e \vee \cdots \vee \xi^m_e .  
\end{align*}
Note that $X_e^0, \xi_e^1, \widetilde{\xi}^2_e$ are mutually independent for each $e$ and
\[
        X^0_e\sim \Bern(p),
        ~~~~~~~~~~~~
        \xi^1_e\sim \Bern(s),
        ~~~~~~~~~~~
         \widetilde{\xi}^2_e\sim \Bern(1 - (1-s)^{m-1}).
\]
Equivalently, the joint distribution of $(X^1,Y)$ is the same as starting with the \erdosrenyi$(n,p)$ incidence matrix $X^0,$  subsampling using $s$ to get $X^1$ and independently subsampling it again using $s'=1-(1-s)^{m-1}$ to get a second graph and applying a uniform random permutation to the second graph to get $Y.$

The information theoretic achievability and converse bounds in the literature for both alignment and detection for two correlated \erdosrenyi\ graphs involve $s^2$, at least in the regime $s^2\to 0$.
It is shown in Section \ref{sec:converses_asymmetric} that those converses extend, with minor changes to the proofs, to the asymmetric  \erdosrenyi\ problems with  parameters $n,s_1,s_2$ by replacing $s^2$ by $s_1s_2.$ This is due to the fact that a key role is played by the intersection graph with the incidence matrix $ X^1_e \wedge X^2_{\pistar(e)}$ which is an \erdosrenyi$(n,ps_1 s_2)$ graph.
Using that method we present corollaries of \Cref{thm:recovery_genie} for $m\geq 2$ \erdosrenyi\ graphs in \Cref{sec:ERalignment_converses}.

\section{Last-matching converse method for detection}
\label{sec:genie-detection}
We show in this section that a similar last-matching method can be used to deduce converses for weak detection.  As for the case of alignment, the idea for detection is to leverage converses for $m=2$ to get converses for $m\geq 3.$  The method uses induction on $m$ and we provide here a proposition for the induction step from $m-1$ to $m$ for $m\geq 3.$  The following three lemmas with basic properties of $\TV$ distance will be used.

\medskip
\begin{lemma} \label{lemma:var_same_conditional}
    Suppose $P_{XZ}$ and $Q_{XZ}$ are two joint probability distributions for random variables $X,Z.$  Let $P_X$ and $Q_X$ denote the corresponding marginal probability distributions of $X$.  If the law of $(Z|X)$ (i.e. the conditional distribution of $Z$ given $X$) is the same for $P$ and $Q$ then $\TV(P_{XZ}, Q_{XZ}) = \TV(P_X,Q_X).$  
\end{lemma}

\medskip

\begin{lemma}  \label{lemma:var_conditional}
    For any probability measure $P$ and event $\calE$, the distance between $P$ and the conditional distribution of $P$ given $\calE$ satisfies: 
    \(  
        \TV(P,P(\cdot |\calE))\leq P(\calE^c).
    \)
\end{lemma}

\medskip

\begin{lemma}   \label{lemma:var_extend}
    Let $P_X$ and $Q_X$ be probability distributions for a random vector $X$.  These distributions can be extended to joint distributions $P$ and $Q$ for random variables $S,X$ such that:
    \begin{itemize}
    \item   $P(S\in \{0,1\}) = Q(S\in \{0,1\}) = 1.$
    \item   $P(S=0)=Q(S=0)=\TV(P_X,Q_X) = \TV(P,Q)$
    \item   The marginal distribution of $X$ under $P$ is $P_X,$
    \item   The marginal distribution of $X$ under $Q$ is $Q_X $
    \item   $ P_{X|S=1} = Q_{X|S=1} $      (equality of conditional distributions.)
    \end{itemize}
\end{lemma}

Using \Cref{lemma:var_same_conditional,lemma:var_extend} we shall consider extensions of $\P$ and $\Q$ to a larger ensemble of random objects and continue to use $\P$ and $\Q$ to denote the extensions.
For a given subset of the objects we denote the marginal distribution of that set of objects under $\P$ or $\Q$ by using either $\P$ or $\Q$ with the objects as subscripts with no commas.  For example, $\P_{X^2X^3\pi_{23}^*}$ denotes the marginal distribution of $(X^2,X^3,\pi_{23}^*)$ under probability distribution $\P$.

Although we will frame the theorem in terms of $\TV$ distances, the intuitive idea is that we suppose a genie provides the decision maker an alignment of matrices $X^2, \cdots , X^m$ which under $H_1$ is the true alignment.  

We continue to use the notation $X^{2:m} = (X^2,\cdots , X^m)$
and $\pistar_{2:m} = (\pistar_{k\ell}:2\leq k,\ell \leq m).$ 
Recall that the joint distribution of $(X,\pistar)$ is defined under $\P$ while $\Q$ only specifies a distribution for $X.$
We extend the probability distribution $\Q$ to a joint distribution of $(X, \pistar_{2:m})$ by requiring equality of conditional distributions: $\Q_{\pistar_{2:m}|X}  =  \Q_{\pistar_{2:m}|X^{2:m}}  = \P_{\pistar_{2:m}|X^{2:m}}.$    Under $\Q$,
$X^1$ is independent of $X^{2:m}$ and since the conditional distribution
of $\pistar_{2:m}$ given $X$ only depends on $X^{2:m}$, it follows that
$X^1$ is independent of $(X^{2:m},\pistar_{2:m})$ (under $\Q$).

\Cref{lemma:basic_Markov} plays an important role in this section.  For both the Gaussian and \erdosrenyi\ it identifies a matrix $Y$ as a function of $(X^{2:m},\pistar_{2:m})$ such that the Markov property \Cref{lemma:basic_Markov} holds.

\medskip 

\begin{theorem} \label{thm:detection_genie}
Suppose $m\geq 3$ and let $\epsilon_m = \TV(\P_{X^{2:m}}, \Q_{X^{2:m}}).$ Then 
\begin{align*} 
    \TV(\P_X,\Q_X) \leq \TV(\P_{X^1Y},\Q_{X^1Y})   + 4 \, \epsilon_m  \, .
\end{align*}
\end{theorem}

\medskip

\begin{remark}  \label{remark:on_TV_bnd}
\begin{enumerate}[label=$\mathrm{(\alph*)}$]
    \item 
    To apply \Cref{thm:detection_genie} suitable assumptions on the parameters ($p,s$ or $\rho$) should be in force such that $\epsilon_m =o(1)$ based on an induction hypothesis and then even stronger conditions (i.e. even smaller upper bounds on $p,s$ or $\rho$) can be sought to make $\TV(\P_{X^1Y},\Q_{X^1Y}) = o(1).$ This latter quantity is the $\TV$ distance for a hypothesis testing problem (asymmetric in the \erdosrenyi\ case) for two observed matrices, namely $X^1$ and $Y.$   See the proof of \Cref{cor:Gaussian_detection_converse} below for details.
    \item To gain some intuition, note that since increasing the number of variables cannot decrease $\TV$ distance, we have $\TV(\P_X,\Q_X) \leq \TV(\P_{X \pistar_{2:m}}, \Q_{X \pistar_{2:m}}).$  Furthermore, if $\epsilon_m=0$ then the conditional distribution of $(X^{2:m},\pistar_{2:m})$ given $(X^1,Y)$ would be the same under $\P$ and $\Q$, implying, by Lemma \ref{lemma:var_same_conditional} with $Z=(X^{2:m},\pistar_{2:m})$,
    that $\TV(\P_{X\pistar_{2:m}},\Q_{X\pistar_{2:m}}) = \TV(\P_{X^1Y},\Q_{X^1Y}).$
    \item The idea for the proof of \Cref{thm:detection_genie} in terms of a decision maker faced with the hypothesis testing problem is the following.   Even if there were a genie available that revealed to the decision maker how matrices $X^2$ through $X^m$ were possibly aligned, the decision maker would not do better than guessing which hypothesis is true based on the information from the genie and the $m$ matrices.  A key idea is that in case the null hypothesis is true, the genie should reveal a plausible alignment of matrices $X^2$ through $X^m.$  That idea motivated our choice of joint distribution for $X$ and $\pistar_{2:m}$ under $\Q$ defined above.
\end{enumerate}
\end{remark}

\begin{proof}  In this paragraph we briefly consider extensions of $\P$ and $\Q$ different from above, so we will use $\widetilde{\P}$ and $\widetilde{\Q}$ for these extensions.
Focus on the distribution of $X^{2:m}$ under $\P$ and $\Q$.
By \Cref{lemma:var_extend} we can extend $\P_{X^{2:m}}$ and $\Q_{X^{2:m}}$ to $\widetilde{\P}$ and $\widetilde{\Q}$ so that there is a binary random variable $S$ jointly distributed with $X^{2:m}$ so that
\begin{align}  \label{eq:two_conditions}
\widetilde{\P}(S=0)=\widetilde{\Q}(S=0)=\epsilon_m \hspace{.2cm} \mbox{and} \hspace{.2cm} \widetilde{\P}_{X^{2:m}|S=1}=\widetilde{\Q}_{X^{2:m}|S=1}.
\end{align} Equivalently, there exist choices of conditional probability distributions $\widetilde{\P}_{S|X^{2:m}}$ and $\widetilde{\Q}_{S|X^{2:m}}$ so that when $\widetilde{\P}_{X^{2:m}}$ and $\widetilde{\Q}_{X^{2:m}}$ are extended using those conditional distributions the properties in \eqref{eq:two_conditions} hold.

Using the conditional probability distributions defined in the previous paragraph (thinking of them as channels as in multiple user information theory), $S$ can also be adjoined to the larger ensemble $(X^1,Y,X^{2:m},\pistar_{2:m})$ under $\P$ and $\Q$ such that \eqref{eq:two_conditions} holds with $\widetilde{\P}$ and $\widetilde{\Q}$ replaced by $\P$ and $\Q$ and under either $\P$ or $\Q$: 
\[
    X^1 \text{ --- } Y \text{ --- } (X^{2:m},\pistar_{2:m}) \text{ --- } X^{2:m}  \text{ --- } S
\]
is a Markov sequence.  And under $\Q$:  $X^1$ is independent of $(Y,X^{2:m},\pistar_{2:m},S).$

By the choice of $\pistar_{2:m}$ and $S$ and the fact that $Y$ is a function of $(X^{2:m},\pistar_{2:m})$ it follows that the law of $(Y,\pistar_{2:m}, X^{2:m}|S=1)$ is the same under $\P$ and $\Q$.   Therefore, the law of $(\pistar_{2:m}, X^{2:m}|S=1,Y)$ is also the same under $\P$ and $\Q$.   By the Markov property (under $\P$) and independence property (under $\Q$) discussed in the previous paragraph, adding in conditioning on $X^1$ does not change the conditional distributions. In other words, the law of
$(\pistar_{2:m}, X^{2:m}|S=1,Y,X^1)$   is the same under $\P$ and $\Q$:
\begin{align}   \label{eq:dist_identity}
\P_{\pistar_{2:m}, X^{2:m}|S=1,Y,X^1} = \Q_{\pistar_{2:m}, X^{2:m}|S=1,Y,X^1}.
\end{align}

With the above preparations, we now have the string of inequalities:
\begin{align*}
\TV(\P_X,\Q_X) &
\stackrel{\text{(a)}}{\leq}  \TV(\P_{XY\pistar_{2:m}},\Q_{XY\pistar_{2:m}})    \\
& \stackrel{\text{(b)}}{\leq}  \TV(\P_{XY\pistar_{2:m}|S=1},\Q_{XY\pistar_{2:m}|S=1}) + 2\epsilon_m \\
& \stackrel{\text{(c)}}{=}  \TV(\P_{X^1Y|S=1},\Q_{X^1Y|S=1}) + 2\epsilon_m  \\
& \stackrel{\text{(d)}}{\leq}  \TV(\P_{X^1Y|S=1},\P_{X^1Y}) + \TV(\P_{X^1Y},\P_{X^1}\otimes \P_{Y} )\\
& ~~~~  +   \TV(\P_{X^1}\otimes \P_{Y},\P_{X^1} \otimes \P_{Y|S=1}) + 2\epsilon_m  \\
& \stackrel{\text{(e)}}{\leq}  \TV(\P_{X^1Y},\P_{X^1}\otimes \P_{Y} ) + 4 \epsilon_m
\end{align*}
where (a) follows because including more variables cannot decrease variational distance, (b) follows by the triangle inequality of $\TV$ and two applications of \Cref{lemma:var_conditional}, (c) follows from \Cref{lemma:var_same_conditional} and \eqref{eq:dist_identity},~(d) follows from the triangle inequality for variational distance and the fact $\Q_{X^1Y|S=1} = \P_{X^1} \otimes \P_{Y|S=1}$, and (e) follows by applying \Cref{lemma:var_same_conditional} to get $ \TV(\P_{X^1Y|S=1},\P_{X^1Y}) \leq \epsilon_m$ and applying \Cref{lemma:var_same_conditional,lemma:var_conditional} to get:  $\TV(\P_{X^1}\otimes \P_{Y},\P_{X^1} \otimes \P_{Y|S=1})=  \TV(\P_{Y}, \P_{Y|S=1})\leq \epsilon_m$.
\end{proof}

\subsection{Application to the Gaussian model}  \label{sec:Gaussian_detection}

See Section \ref{sec:Gaussian_case_recovery} including the definition of $\rho_m.$   By the converse in~\cite{wu2023testing} for two graphs, weak detection based on observation of $(X^1,Y)$ is impossible if $n\rho_m^2 \leq (4-\epsilon')\log n$ for some $\epsilon' > 0$, or equivalently if
\[
    n\rho^2 \leq \frac{(4-\eps)}{m-1} \log n
\]
for some $\epsilon>0$. This yields the following corollary of the converse in
~\cite{wu2023testing} and \Cref{thm:detection_genie}.

\medskip 

\begin{corollary} \label{cor:Gaussian_detection_converse}
(Gaussian model, $m\geq 2$)
    Let $\epsilon >0$ denote a fixed constant. If
    \begin{align} \label{eq:Gaussian_converse_cond}
        \rho^2 \leq \frac{4-\epsilon}{m-1} \cdot \frac{\log n}{n},
    \end{align}
    then weak detection is impossible.
\end{corollary}
\begin{proof}
We elaborate on Remark \ref{remark:on_TV_bnd}(a) to provide a proof of \Cref{cor:Gaussian_detection_converse} by induction on $m.$   Take the base case to be $m=2$. For $m=2$ the corollary reduces to the converse in ~\cite{wu2023testing}.  So for the sake of proof by induction suppose $m \geq 3$ and that \Cref{cor:Gaussian_detection_converse} is true for $m-1.$  Given $\epsilon > 0$ suppose that \eqref{eq:Gaussian_converse_cond} holds.
Then also
    \[
        \rho^2 \leq \frac{4-\epsilon}{(m-1)-1} \cdot \frac{\log n}{n},
    \]
so by the induction hypothesis, $\epsilon_m := \TV(\P_{X^{2:m}},\Q_{X^{2:m}})= o(1).$
Also, as noted above, \eqref{eq:Gaussian_converse_cond} implies that $\TV(\P_{X^1Y},Q_{X^1Y})=o(1).$ 
Applying \Cref{thm:detection_genie} completes the proof that $\TV(\P_X,\Q_X)=o(1)$ for $m,$  completing the proof by induction.
\end{proof}
This corollary is not tight.  It is shown in \cite{ameenhajekGaussian} that weak detection is impossible if $\rho^2 \leq \frac {8-\epsilon} m \cdot \frac{\log n} n.$
That is the sharp threshold -- \cite{ameenhajekGaussian} also shows that strong detection is possible if $\rho^2 \geq \frac {8+\epsilon} m \cdot \frac{\log n} n.$     Roughly speaking, we can conclude that in the Gaussian case, it is strictly more difficult to detect correlation among $m$ graphs than it is to detect correlation between the first graph and an aligned version of graphs 2 through $m$.

\subsection{Towards application to \erdosrenyi\ model}

Just as for the problem of alignment of multiple \erdosrenyi\ graphs described in \Cref{sec:towards_ER}, applying the last-matching bound for weak detection of correlation in such graphs involves converse bounds for two correlated \erdosrenyi\ graphs with asymmetry.   Therefore, the converse bounds for weak detection of correlation for multiple \erdosrenyi\ graphs are deferred to \Cref{sec:converses_multiple_ER}.

\section{Converses for asymmetric two-graph \erdosrenyi\ model}
\label{sec:converses_asymmetric}

Consider the \emph{asymmetric} two-graph \erdosrenyi\ model with parameters $p,s_1,s_2\in(0,1)$. Under $\P$, there is a uniformly random hidden alignment $\pistar\in\Sn$ and, conditioned on $\pistar$,
\[
        X^1_e=X^0_e\xi^1_e,
        ~~~~~~
        X^2_{\pistar(e)}=X^0_e\xi^2_e,
        ~~~~~~ e\in\En,
\]
where $X^0_e\sim\Bern(p)$, $\xi^1_e\sim\Bern(s_1)$, and
$\xi^2_e\sim\Bern(s_2)$ are mutually independent over $e\in\En$.

\subsection{Converses for weak detection for asymmetric two-graph model}

Consider the detection problem with observation $X$ such that under hypothesis $H_1$ the observation has the distribution $\P$ with parameters $n, p, s_1, s_2$ defined above and under hypothesis $H_0$ the observation has distribution $\Q$ under which  $X^1$ and $X^2$ are independent with the same marginal distributions as under $\P$. 

Let $\varrho(\lambda)$ denote the limiting density of the densest subgraph of an \erdosrenyi$(n,\lambda/n)$ graph, i.e.
\begin{align}   \label{eq:def_varrho}
    \varrho(\lambda) = \lim_{n\to\infty} \, \max_{\emptyset \neq U \subseteq [n]} \frac{|\mathcal{E}(U)|}{|U|} \, .
\end{align}
The following theorem collects the asymmetric weak-detection converses needed for the three regimes
\(
    p=n^{-\alpha+o(1)},
\)
with
\(
    \alpha=1, ~~ 0<\alpha<1
\)
and
\(
    \alpha=0
\)
respectively.

\medskip

\begin{theorem}[Impossibility of weak detection for two asymmetric Erdős--Rényi graphs]
\label{thm:asym_er_weak_detection_converses}
The following converse bounds hold for the asymmetric two-graph Erdős--Rényi detection problem.

\begin{enumerate}
    \item[$\mathrm{(i)}$] \thmcaseLabel{thm:asym_er_weak_detection_converses-i}
    \emph{(Sparse regime.)}
    Suppose
    \[
        s_{1}s_{2} \to 0
        ~~~~
        \mbox{and}
        ~~~~
        np s_1 s_2 \leq 1-\omega(n^{-1/3}) \, .
    \]
    Then
    \(
        \TV(\P,\Q) = o(1) ,
    \)
    i.e. weak detection is impossible.

    \item[$\mathrm{(ii)}$] \thmcaseLabel{thm:asym_er_weak_detection_converses-ii}
    \emph{(Moderately sparse regime.)}
    Suppose $\alpha \in (0,1)$, $p = n^{-\alpha + o(1)},$ and $\max\{s_1,s_2\} = 1-\Omega(1)$. Also suppose for a fixed $\epsilon>0$
    \[
        np s_1s_2\leq  \varrho^{-1}(1/\alpha) - \epsilon.
    \]
    Then
    \(
        \TV(\P,\Q) = o(1).
    \)

    \item[$\mathrm{(iii)}$] \thmcaseLabel{thm:asym_er_weak_detection_converses-iii}
    \emph{(Dense regime.)}
    Suppose
\(
        p\le 1-\Omega(1), p=n^{-o(1)},
        \max\{s_1, s_2\} \to 0,
\)
and
\(
        np^2s_i=\omega(1)
\)
for $i \in \{1,2\}$.
Also suppose for a fixed $\epsilon>0$
    \begin{align}
        np s_1s_2  \le \frac{ (2-\epsilon) \log n}{ \log \frac{1}{p} -1 + p }.
        \label{eq:lb-denseER}
    \end{align}
    Then
    \(
        \TV\left(\P,\Q\right) = o(1).
    \)
\end{enumerate}
\end{theorem}

\Cref{thm:asym_er_weak_detection_converses} is proved in Appendix \ref{app:detection_converses_asymmetric}.    The sparse-regime converse, \Cref{thm:asym_er_weak_detection_converses-i}, is proved using a combination of ideas from \cite{feng2025strong} and \cite{wu2023testing}. \Cref{thm:asym_er_weak_detection_converses-i} improves on the previous converses by including the asymmetric case and extending the range of $np$ (assumed to be constant in \cite{feng2025strong} and assumed to be $\omega(\log^2 n)$ in \cite{wu2023testing}). Also, \cite{wu2023testing} only considers impossibility of strong detection.
The moderately sparse regime converse, \Cref{thm:asym_er_weak_detection_converses-ii}, is proved by adapting the analysis of~\cite{ding2023detection}, where the sharp two-graph detection threshold is obtained in the symmetric case $s_1=s_2=s$. 
The dense-regime converse,\Cref{thm:asym_er_weak_detection_converses-iii}, is based on the symmetric case converse in~\cite{wu2023testing}.
The proofs all use the conditional second moment method on the likelihood ratio to show impossibility of detection; they differ on the conditioning event used.

\subsection{Converses for alignment for asymmetric two-graph model}

\subsubsection{Partial alignment for two \erdosrenyi\ graphs}

The following theorem collects the asymmetric alignment converses needed for the three regimes
\(
    p=n^{-\alpha+o(1)},
\)
with
\(
    \alpha=1, ~~ 0<\alpha<1
\)
and
\(
    \alpha=0
\)
respectively.

%Ding and Du~\cite{ding2023densesubgraph} proved that, for the symmetric two-graph model, the threshold for partial alignment in the regime $p=n^{-\alpha+o(1)}$ is whether $nps^2$ is greater than or less than  $\lambda_\alpha$, where $\lambda=nps^2$. For our application we need the following asymmetric versions of partial-alignment converses.

\medskip

\begin{theorem}[Asymmetric two-graph partial-alignment converses]
\label{thm:asym_dd_recovery}
%Let \(t:=s_1s_2\).  
The following converse bounds hold for the asymmetric two-graph \erdosrenyi\ alignment problem.

\begin{enumerate}[label=\textup{(\roman*)}]
\item \thmcaseLabel{thm:asym_dd_recovery-i}
\textup{(Sparse regime.)} 
Suppose \(p=n^{-\Omega(1)}\), \(\max\{s_1, s_2\} \to0\), \(np=\omega(\log^2 n)\), and for a fixed \(\epsilon>0\),
\[
        nps_1s_2 \le 1-\epsilon .
\]
Then partial alignment is intractable.

\item \thmcaseLabel{thm:asym_dd_recovery-ii}
\textup{(Moderately sparse regime.)}
Suppose $\alpha \in (0,1)$, $p = n^{-\alpha + o(1)},$ and $\max\{s_1,s_2\}\to 0.$ Also suppose for a fixed $\epsilon>0$
    \[
        np s_1s_2\leq  \varrho^{-1}(1/\alpha) - \epsilon.
    \]
Then partial alignment is intractable.

\item \thmcaseLabel{thm:asym_dd_recovery-iii}
\textup{(Dense regime.)}
Suppose
\(
        p\le 1-\Omega(1), p=n^{-o(1)},
        \max\{s_1, s_2\} \to 0,
\)
and
\(
        np^2s_i=\omega(1)
\)
for $i \in \{1,2\}$.
Also suppose for a fixed $\epsilon>0$
\[
np s_1s_2  \le \frac{ (2-\epsilon) \log n}{ \log \frac{1}{p} -1 + p }. 
\]
Then partial alignment is intractable.
\end{enumerate}
\end{theorem}

\medskip

\Cref{thm:asym_dd_recovery} is proved in Appendix \ref{app:converses_asym_recovery}.
The converse and proof for the moderately sparse regime, \Cref{thm:asym_dd_recovery-ii}, are based on the converse and proof for the symmetric case $s_1=s_2$ in  \cite{ding2023densesubgraph}.   The converses and proofs for the sparse and dense regimes,
\Cref{thm:asym_dd_recovery-i,thm:asym_dd_recovery-iii}, are based on the converses and proofs in the symmetric case from \cite{wu2022settling}. The proof in~\cite{wu2022settling} uses a relationship between the mutual information and a corresponding minimum mean square error, whereas the approach in~\cite{ding2023densesubgraph} directly shows that the posterior distribution is spread out in a way that precludes partial alignment. Interestingly, both approaches rely on the impossibility of weak detection for the corresponding regimes in \Cref{thm:asym_er_weak_detection_converses}.

%Note that in the symmetric case with $s_1 = s_2 = s$, the condition $s\to 0$ is implied by the conditions $p = n^{-\alpha + o(1)}$ and $n p s^2 \leq \lambda_{\alpha} - \epsilon$. The above asymmetric version requires both $s_1 \to 0$ and $s_2 \to 0$.

\subsubsection{Almost-exact alignment for two \erdosrenyi\ graphs}

The following is an immediate consequence of \cite[Theorem 2]{cullina2019kcore}. 
\medskip 

\begin{theorem}[Asymmetric two-graph almost-exact alignment converse]
    \label{thm:almost_exact_recovery}
    If $nps_1s_2 = O(1)$ then almost-exact alignment is impossible.
\end{theorem}
A simple explanation of \Cref{thm:almost_exact_recovery} is that the condition implies that the number of isolated vertices in the intersection graph $(X_e^1\wedge X^2_{\pistar(e)})_{e\in \calE}$ grows at most linearly in $n.$   If $nps_1s_2 \to \infty$ then under mild additional assumptions (e.g. $p\leq n^{-\Omega(1)}$ and $ps_1s_2 \leq \frac 1 {8e^3}$) almost-exact alignment is possible \cite{cullina2019kcore}.

\subsubsection{Exact alignment for two \erdosrenyi\ graphs}

\begin{theorem}[Asymmetric two-graph exact-alignment converse]
\label{thm:asym_exact_recovery}
Suppose that $p\max\{s_1,s_2\}=o(1)$ and let $\epsilon > 0.$   If
 \[
        s_1s_2 p(1-\sqrt{p})^2 \geq (1+\epsilon)\frac{\log n} n
\]
then exact alignment is possible and if
 \[
        s_1s_2 p(1-\sqrt{p})^2 \leq (1 - \epsilon)\frac{\log n} n
\]
then exact alignment is impossible.
\end{theorem}
\begin{proof}
\Cref{thm:asym_exact_recovery} follows directly from \cite{wu2022settling}, which improves on earlier results of \cite{cullina2017exact}.  Indeed,  \cite{wu2022settling} establishes that with $p_{ij} := \P(X^1_e = i, X^2_{\pi(e)} =j),$  the critical threshold for exact alignment is 
$(\sqrt{p_{00}p_{11}} - \sqrt{p_{01}p_{10}} )^2 \sim {\log (n)}/ n.$
Substituting
\[
p_{00} = 1-p + p(1-s_1)(1-s_2),
~~~~
p_{01} = p(1-s_1)s_2,
~~~~
p_{10} = ps_1(1-s_2), 
~~~~
p_{11} = ps_1s_2
\]
yields
\begin{align*}
    &(\sqrt{p_{00}p_{11}} - \sqrt{p_{01}p_{10}} )^2  \\
    & ~~~~ 
    = ps_1s_2\left[
    1 + p(1-2s_1 - 2s_2 + 2s_1s_2) -2\sqrt{p(1-s_1)(1-s_2)(1-ps_1-ps_2+ps_1s_2)}
    \right] \, .
\end{align*}
The assumption $p\max\{s_1,s_2\}=o(1)$ implies that either $p\to 0$ or $\max\{s_1,s_2\} \to 0.$  In either case, $(\sqrt{p_{00}p_{11}} -  \sqrt{p_{01}p_{10}} )^2 \sim s_1s_2 p(1-\sqrt{p})^2,$ implying the theorem.
\end{proof}

\section{Converses for multiple correlated \erdosrenyi\ graphs}
\label{sec:ER_corollaries}

\subsection{Converses for detection of multiple correlated \erdosrenyi\ graphs}
\label{sec:converses_multiple_ER}

In this section we combine \Cref{thm:detection_genie,thm:asym_er_weak_detection_converses} and proof by induction on $m$ as in the proof of \Cref{cor:Gaussian_detection_converse} to yield converses for weak detection of correlation among multiple graphs for the \erdosrenyi\ model. This provides converses for the $m$-graph model with parameters $n,m,p,s$
by applying the converses for the asymmetric two graph model with parameters $n,p,s_1=s,s_2=1-(1-s)^{m-1}.$  The two-graph converses in Theorems 5.1 and 5.2 only cover regimes such that $\max\{s_1,s_2\}\to 0$ or $s_1s_2\to 0$ so the corresponding converses entail $s\to 0.$   In that case, $s_2 \sim (m-1)s.$  

Those two theorems for the dense regime include the assumption $np^2s_i=\omega(1).$   As noted in the proof in Appendix B.3, without loss of generality, it can be assumed that the threshold condition holds with equality, so we can assume in the dense regime for multiple graphs that~\eqref{eq:dense_nec} below holds with equality.   Then $s=n^{-1/2 + o(1)}$ in which case the assumptions $\max\{s_1,s_2\}\to 0$ and
$np^2s_i=\omega(1)$ for $i=1,2$ are satisfied and need not be imposed in the corollary. 
\medskip 

\begin{corollary}[Impossibility of weak detection for multiple ER graphs]  
\label{cor:asym_er_weak_detection_converses}
The following converse bounds hold for the detection problem with $m$ Erdős--Rényi graphs.

\begin{enumerate}
    \item[$\mathrm{(i)}$] \thmcaseLabel{cor:asym_er_weak_detection_converses-i}
    \emph{(Sparse regime.)}
    Suppose
    \[
        s \to 0
        ~~~~
        \mbox{and}
        ~~~~
        np s^2 \leq \frac{1-\omega(n^{-1/3})}{m-1} \, .
    \]
    Then
    \(
        \TV(\P,\Q) = o(1) ,
    \)
    i.e. weak detection is impossible. 

    \item[$\mathrm{(ii)}$] \thmcaseLabel{cor:asym_er_weak_detection_converses-ii}
    \emph{(Moderately sparse regime.)}
    Suppose $\alpha \in (0,1)$, $p = n^{-\alpha + o(1)}$ and for some fixed $\epsilon>0$,
    \begin{align*}  
        np s^2 \leq \frac{ \varrho^{-1}(1/\alpha) - \epsilon }{m-1}.
    \end{align*}
    Then
    \(
        \TV(\P,\Q) = o(1).
    \)

    \item[$\mathrm{(iii)}$] \thmcaseLabel{cor:asym_er_weak_detection_converses-iii}
    \emph{(Dense regime.)}
    Suppose  $p \le 1- \Omega(1),$  $p=n^{-o(1)},$ and for some fixed $\epsilon>0$,
    \begin{align}  \label{eq:dense_nec}
        np s^2  \le \frac{ (2-\epsilon) \log n}{(m-1)\left( \log \frac{1}{p} -1 + p \right)}. 
    \end{align}
    Then
    \(
        \TV\left(\P,\Q\right) = o(1).
    \)
\end{enumerate}
\end{corollary}
While Corollary \ref{cor:asym_er_weak_detection_converses}(i) does not impose a separate condition on $p$ it is listed as ``sparse regime" because it includes the case $p=n^{-1+o(1)}.$   For larger values of $p$ the other parts of Corollary \ref{cor:asym_er_weak_detection_converses} are stronger.   Since $\varrho(1)=\varrho^{-1}(1)=1,$ the conditions in Corollary \ref{cor:asym_er_weak_detection_converses}(i) can be viewed as a limiting version of Corollary \ref{cor:asym_er_weak_detection_converses}(ii) as $\alpha \to 1.$

In the positive direction for the sparse regime, the recent report \cite{ochoa2026detection} applies a method inspired by \cite{ding2023detection} to derive positive results for strong detection in case $m\geq 2.$ See \Cref{sec:discussion_detection}.

\subsection{Converses for alignment of multiple correlated \erdosrenyi\ graphs}
\label{sec:ERalignment_converses}

\subsubsection{Partial alignment for multiple \erdosrenyi\ graphs}

Combining \Cref{thm:recovery_genie,thm:asym_dd_recovery}
yields the following corollary.

\medskip

\begin{corollary}[Intractability of partial alignment for multiple ER graphs]
\label{cor:asym_dd_recovery} 
The following converse bounds hold for the alignment problem with $m$
\erdosrenyi\  graphs.

\begin{enumerate}[label=\textup{(\roman*)}]
\item \thmcaseLabel{cor:asym_dd_recovery-i}
\textup{(Sparse regime.)}
Suppose \(p=n^{-\Omega(1)}\), \(np=\omega(\log^2 n)\), and, for
some fixed \(\epsilon>0\),
\[
        nps^2 \le \frac{1-\epsilon} {m-1} .
\]
Then partial alignment is intractable.

\item \thmcaseLabel{cor:asym_dd_recovery-ii}
\textup{(Moderately sparse regime.)}
Suppose that for some $\alpha \in (0,1)$
\(
        p=n^{-\alpha+o(1)},
\)
and that, for some fixed \(\epsilon>0\),
\[
        nps^2 \le \frac{\varrho^{-1}(1/\alpha)-\epsilon}{m-1} .
\]
Then partial alignment is intractable.

\item \thmcaseLabel{cor:asym_dd_recovery-iii}
\textup{(Dense regime.)}
    Suppose  $p \le 1- \Omega(1),$  $p=n^{-o(1)},$ and for some fixed $\epsilon>0$,
    \begin{align*}
        np s^2  \le \frac{ (2-\epsilon) \log n}{(m-1)\left( \log \frac{1}{p} -1 + p \right)}. 
    \end{align*}
Then partial alignment is intractable.
\end{enumerate}
\end{corollary}

The threshold in \Cref{cor:asym_dd_recovery} in case $np$ is fixed (special case of sparse regime) matches the threshold in \cite{vassaux2025} for intractability of partial alignment, which is conjectured therein to be tight.  It is possible that the threshold identified in \cite{ochoa2026detection} for achievability is the tight condition.

\subsubsection{Almost-exact alignment for multiple \erdosrenyi\ graphs}

The natural corollary to \Cref{thm:recovery_genie,thm:almost_exact_recovery} is that almost-exact alignment is impossible if $nps(1-(1-s)^{m-1})= O(1).$   However, by the fact $s \leq 1 - (1-s)^{m-1} \leq (m-1)s$, it follows that almost-exact alignment for $m$ graphs is intractable if $nps^2 = O(1).$  We thus get the following corollary with condition not depending on $m.$

\medskip

\begin{corollary}[Impossibility of almost-exact alignment for multiple ER graphs]
    \label{cor:almost_exact_recovery}
    If $nps^2 = O(1)$ then almost-exact alignment is impossible.
\end{corollary}
This converse is tight under the additional assumptions $p\leq n^{-\Omega(1)}$ and $ps^2 \leq \frac 1 {8e^3}$ because if $nps^2 \to \infty$ under such assumptions then the graphs can be aligned sequentially pairwise to achieve almost-exact alignment.

\subsubsection{Exact alignment for multiple \erdosrenyi\ graphs}

Combining \Cref{thm:recovery_genie,thm:asym_exact_recovery}
yields the following corollary.

\medskip

\begin{corollary}[Impossibility of exact-alignment for multiple ER graphs]
\label{cor:asym_exact_recovery}
Suppose that $ps=o(1)$ and let $\epsilon > 0.$   If
\[
        nps\big(1-(1-s)^{m-1}\big)(1-\sqrt{p})^2 \leq (1-\epsilon)\log n
\]
then exact alignment is impossible.  (It follows that if
\(
        nps^2(1-\sqrt{p})^2 \leq \frac{(1-\epsilon)\log n} {m-1}
\)
then exact alignment is impossible.)
\end{corollary}
This is tight for $p = C \log(n)/n$, by the positive result of~\cite{ameen2024exact,racz2024harnessing}.

\section{Discussion} \label{sec-discussion}

In this section we compare the converses of this paper to existing converses in the literature.  As noted in \Cref{sec:Gaussian_case_recovery,sec:Gaussian_detection}, a single
sharp threshold $\rho^2 \sim \frac 8 m \cdot \frac {\log n} n$ for multiple correlated Gaussian graphs has been previously identified for weak to strong detection and intractability of partial alignment to exact alignment.    The last-matching method of this paper yields converses for  $\rho^2 \leq \frac {4(1-\epsilon)} {m-1} \cdot \frac {\log n} n,$ which is smaller than the actual threshold by the factor $\frac{m}{2(m-1)}.$

In the case of two graphs, all the converse bounds in \Cref{sec:converses_asymmetric} with possibly different subsampling probabilities $s_1$ and $s_2$ reduce to previously known bounds in the special case $s_1=s_2,$ for which case those bounds are known to be tight.   Therefore, the resulting converse bounds in \Cref{sec:ER_corollaries} are tight in the special case $m=2.$   Thus, for the remainder of this section we focus on the \erdosrenyi\ models for three or more graphs.

\subsection{Detection thresholds}  \label{sec:discussion_detection}

We first comment on other {\em converses} in the literature for detection of correlation for three or more correlated \erdosrenyi\ graphs.  \cite{vassaux2025} proved that if $np$ is fixed as $n\to\infty$ and $nps(1-(1-s)^{m-1}) < 1$ then partial alignment is intractable.   This is the same threshold as in Corollary \ref{cor:asym_dd_recovery}(i) although  Corollary \ref{cor:asym_dd_recovery}(i) holds under conditions that don't overlap with the condition that $np$ is fixed.  We know of no other converses for detection of correlation for three or more correlated \erdosrenyi\ graphs.

We next comment on {\em positive results} about detection of correlation for three or more correlated \erdosrenyi\ graphs.   The paper \cite{ochoa2026detection} gives a positive result for both the sparse and moderately sparse regimes, as follows.  
Let 
\[
    q_m(s) \coloneqq 1 - (1-s)^m - ms(1-s)^{m-1},
\]
which is the probability that an edge in the parent graph $X^0$ is included in at least two of the $m$ observed graphs.   \cite{ochoa2026detection} shows that if $p=n^{-\alpha + o(\alpha)}$ for some $\alpha \in (0,1]$, if $nps^2 = O(1),$ and $npq_m(s)\geq \varrho^{-1}\left(\frac{m-1}{\alpha}\right)(1+\epsilon)$, then strong detection is possible.    The positive result of \cite{ochoa2026detection} is for strong detection while Corollary \ref{cor:asym_er_weak_detection_converses} is for weak detection, so there is the possibility of different thresholds for these two properties.   
Still, we would like to compare the sufficient condition $npq_m(s)\geq \varrho^{-1}\left(\frac{m-1}{\alpha}\right)$ of \cite{ochoa2026detection} to the converse condition $nps^2 \leq \frac{\varrho^{-1}(1/\alpha)}{m-1}.$  We consider fixed $m\geq 3$ and, for brevity, ignore $\epsilon$ terms.   
Note that $q_m(s) \leq \binom m2 s^2$, while if $s\to 0$, then $q_m(s) \sim \binom m2 s^2$.  Thus, when $s\to 0$, the sufficient condition is asymptotically equivalent to
$nps^2 \geq \frac{1}{\binom m2}\varrho^{-1}\left({(m-1)}/{\alpha}\right).$

Consider the moderately sparse regime and a small value of $\alpha.$   Since  ${\varrho(\lambda)}/{\lambda}$ decreases from $1$ to $1/2$ as $\lambda$ increases from $1$ to $\infty,$  it holds for small $\alpha$ that 
\[
    \varrho^{-1}\Big(\frac{m-1}{\alpha}\Big) \approx (m-1)\varrho^{-1}\Big(\frac 1 {\alpha}\Big)
\]
and the ratio of the bounds for small $s$ and $\alpha$ is approximately $\frac m {2(m-1)}$ (the same ratio found for the Gaussian case).
Since $\varrho^{-1}(x) \leq 2x$, a still stronger sufficient condition for strong detection is $nps^2 \geq \frac{4}{\alpha m}.$  For the sparse case $\alpha=1$,
$\varrho^{-1}(\alpha)=1$ so the converse condition of Corollary \ref{cor:asym_er_weak_detection_converses}(i)  is $nps^2 \leq \frac 1 {m-1}$.
So for the sparse case $p=n^{-1+o(1)}$ with $s\to 0$ the ratio between the achievable and converse bounds is $\frac m {4(m-1)} \approx  \frac 1 4$ for large $m$.

The following positive result for strong detection in the dense case is proved in Appendix~\ref{app:ER-GLRT-achievability} (it is similar to the proof for $m=2$ in \cite{wu2023testing}).

\medskip 

\begin{theorem}[Achievability of strong detection in dense regime] 
\label{thm:ER-GLRT-achievability}
Suppose that
\(
    p=n^{-o(1)}
\)
and
\(
    p\le 1-c
\)
for some fixed constant \(c>0\). 
If, for some fixed $\epsilon > 0$,
\[
    (n-1) \,  p s^2
    \ge (1+\epsilon) \frac{4\log n} {m\, \big( \log(1/p) -1+p  \big)} \, , 
    ~~~~ \mbox{ and } ~~~~
    \sum_{r=3}^m
    s^{r-2}p^{-\binom{r-1}{2}}
    \, = \, 
    o(1) \, ,
\]
then strong detection is achievable.
\end{theorem}
Corollary \ref{cor:asym_er_weak_detection_converses}(iii) states the converse:
Suppose
\(
        p\le 1-\Omega(1), p=n^{-o(1)},
\)
and
\(
        np^2s=\omega(1).
\)
If, for some fixed \(\epsilon>0\),
\[
        nps^2
        \le \frac{ (2-\epsilon)\log n}{(m-1)(\log(1/p)-1 + p)},
\]
then $\TV(\P.\Q)=o(1)$ (i.e. weak detection is impossible).   The thresholds in these two results differ by the factor $\frac m {2(m-1)}$ (the same ratio found for the Gaussian case).

{\bf Open problem 1:} It is an open problem to determine whether any of the three converse bounds for impossibility of weak detection given in Corollary \ref{cor:asym_er_weak_detection_converses} are tight.

As noted above, Corollary \ref{cor:almost_exact_recovery} for impossibility of almost-exact alignment is tight under two additional assumptions 
$p\leq n^{-\Omega(1)}$ and $ps^2 \leq \frac 1 {8e^3}$ because if $nps^2 \to \infty$ under such assumptions then the graphs can be aligned sequentially pairwise to achieve almost-exact alignment.

\subsection{Alignment thresholds}

The sharp threshold for exact alignment in the specific sparse regime $p=c\log (n)/n$ for a fixed constant $c$ was identified in~\cite{ameen2024exact,racz2024harnessing} (and also in~\cite{ameen2024aligning}.)
The sharp condition is $cs(1-(1-s)^{m-1}) = 1.$    Thus, in that specific regime, Corollary \ref{cor:asym_exact_recovery} is tight up to the exact constant. 

{\bf Open problem 2:} It is an open problem to determine whether Corollary \ref{cor:asym_exact_recovery} is tight for a broader range of $p.$

{\bf Open problem 3:} It is an open problem to determine whether any of the three converse bounds for intractability of partial alignment given in Corollary \ref{cor:asym_dd_recovery} are tight.

\section*{Acknowledgments}
The authors acknowledge the use of AI-based tools for assistance with literature search, editorial refinement, and improving the clarity of exposition. In addition, AI-based tools were used to assist in assembling the existing arguments from the symmetric setting, which was a first step in our process of writing Appendices~\ref{app:asym_common_tools}--\ref{app:asym_alignment_proofs}. The authors take full responsibility for the correctness of the results.

\bibliographystyle{alpha}
\bibliography{bibliography.bib}

@article{deshpande2017asymptotic,
  title={Asymptotic mutual information for the balanced binary stochastic block model},
  author={Deshpande, Yash and Abbe, Emmanuel and Montanari, Andrea},
  journal={Information and Inference: A Journal of the IMA},
  volume={6},
  number={2},
  pages={125--170},
  year={2017},
  publisher={Oxford University Press}
}

@article{ameenhajekGaussian,
      title={Sharp Detection Threshold for Correlation among Multiple Unlabeled {G}aussian Networks}, 
      author={Taha Ameen and Bruce Hajek},
      year={2026},
      journal ={arXiv preprint arXiv:2504.16279},
      url={https://arxiv.org/abs/2504.16279}, 
}

@article{racz2024harnessing,
	author = {R{\'a}cz, Mikl{\'o}s Z and Zhang, Jifan},
	journal = {Advances in Neural Information Processing Systems},
	pages = {53834--53886},
	title = {Harnessing multiple correlated networks for exact community recovery},
	volume = {37},
	year = {2024}}

@inproceedings{ameen2024exact,
	author = {Ameen, Taha and Hajek, Bruce},
	booktitle = {2025 IEEE International Symposium on Information Theory (ISIT)},
	doi = {10.1109/ISIT63088.2025.11195705},
	pages = {1-6},
	title = {Exact Random Graph Matching with Multiple Graphs},
	year = {2025}}

@article{ameen2024aligning,
    author = {Ameen, Taha and Hajek, Bruce},
    title = {Aligning Multiple Inhomogeneous Random Graphs: Fundamental Limits of Exact Recovery},
    journal = {Operations Research},
    volume = {0},
    number = {0},
    year = {0},
    doi = {10.1287/opre.2025.1808},
    URL = {https://doi.org/10.1287/opre.2025.1808},
    eprint = {https://doi.org/10.1287/opre.2025.1808}
}

@inproceedings{ameen2025detecting,
	author = {Ameen, Taha and Hajek, Bruce},
	booktitle = {2025 IEEE International Symposium on Information Theory (ISIT)},
	doi = {10.1109/ISIT63088.2025.11195646},
	pages = {1-6},
	title = {Detecting Correlation Between Multiple Unlabeled {G}aussian Networks},
	year = {2025}}

@article{calissano2024graph,
	author = {Calissano, Anna and Papadopoulo, Theodore and Pennec, Xavier and Deslauriers-Gauthier, Samuel},
	journal = {Human Brain Mapping},
	number = {1},
	pages = {e26554},
	publisher = {Wiley Online Library},
	title = {Graph alignment exploiting the spatial organization improves the similarity of brain networks},
	volume = {45},
	year = {2024}}

@article{cullina2016improved,
	author = {Cullina, Daniel and Kiyavash, Negar},
	journal = {ACM SIGMETRICS Performance Evaluation Review},
	number = {1},
	pages = {63--72},
	publisher = {ACM New York, NY, USA},
	title = {Improved achievability and converse bounds for {E}rd{\H{o}}s-{R}{\'e}nyi graph matching},
	volume = {44},
	year = {2016}}

@article{cullina2017exact,
	author = {Cullina, Daniel and Kiyavash, Negar},
	journal = {arXiv preprint arXiv:1711.06783},
	title = {Exact alignment recovery for correlated {E}rd{\H{o}}s-{R}{\'e}nyi graphs},
	year = {2017}}

@article{cullina2019kcore,
	author = {Cullina, Daniel and Kiyavash, Negar and Mittal, Prateek and Poor, Vincent},
	journal = {Proceedings of the ACM on Measurement and Analysis of Computing Systems},
	number = {3},
	pages = {1--21},
	publisher = {ACM New York, NY, USA},
	title = {Partial recovery of {E}rd{\H{o}}s-{R}{\'e}nyi graph alignment via $k$-core alignment},
	volume = {3},
	year = {2019}}

@article{ding2023densesubgraph,
	author = {Ding, Jian and Du, Hang},
	journal = {The Annals of Statistics},
	number = {4},
	pages = {1718--1743},
	publisher = {Institute of Mathematical Statistics},
	title = {Matching recovery threshold for correlated random graphs},
	volume = {51},
	year = {2023}}

@article{ding2023detection,
	author = {Ding, Jian and Du, Hang},
	journal = {IEEE Transactions on Information Theory},
	number = {8},
	pages = {5289--5298},
	publisher = {IEEE},
	title = {Detection threshold for correlated {E}rd{\H{o}}s-{R}{\'e}nyi graphs via densest subgraph},
	volume = {69},
	year = {2023}}

@article{du2025optimal,
	author = {Du, Hang},
	journal = {arXiv preprint arXiv:2502.12077},
	title = {Optimal recovery of correlated {E}rdos-{R}enyi graphs},
	year = {2025}}

@article{even2025statistical,
	author = {Even, Bertrand and Ganassali, Luca},
	journal = {arXiv preprint arXiv:2512.00610},
	title = {Statistical-computational gap in multiple {G}aussian graph alignment},
	year = {2025}}

@inproceedings{ganassali2022sharp,
	author = {Ganassali, Luca},
	booktitle = {Mathematical and Scientific Machine Learning},
	organization = {PMLR},
	pages = {314--335},
	title = {Sharp threshold for alignment of graph databases with {G}aussian weights},
	year = {2022}}

@article{hall2023partial,
	author = {Hall, Georgina and Massouli{\'e}, Laurent},
	journal = {Operations Research},
	number = {1},
	pages = {259--272},
	publisher = {INFORMS},
	title = {Partial recovery in the graph alignment problem},
	volume = {71},
	year = {2023}}

@inproceedings{josephs2021recovery,
	author = {Josephs, Nathaniel and Li, Wenrui and Kolaczyk, Eric. D.},
	booktitle = {2021 55th Asilomar Conference on Signals, Systems, and Computers},
	doi = {10.1109/IEEECONF53345.2021.9723092},
	pages = {1268-1273},
	title = {Network Recovery from Unlabeled Noisy Samples},
	year = {2021}}

@inproceedings{narayanan2009deanonymizing,
	author = {Narayanan, Arvind and Shmatikov, Vitaly},
	booktitle = {2009 30th IEEE Symposium on Security and Privacy},
	organization = {IEEE},
	pages = {173--187},
	title = {De-anonymizing social networks},
	year = {2009}}

@inproceedings{narayanan2008robust,
	author = {Narayanan, Arvind and Shmatikov, Vitaly},
	booktitle = {2008 IEEE Symposium on Security and Privacy (sp 2008)},
	organization = {IEEE},
	pages = {111--125},
	title = {Robust de-anonymization of large sparse datasets},
	year = {2008}}

@inproceedings{pedarsani2011privacy,
	author = {Pedarsani, Pedram and Grossglauser, Matthias},
	booktitle = {Proceedings of the 17th ACM SIGKDD International Conference on Knowledge Discovery and Data Mining},
	pages = {1235--1243},
	title = {On the privacy of anonymized networks},
	year = {2011}}

@article{singh2008global,
	author = {Singh, Rohit and Xu, Jinbo and Berger, Bonnie},
	journal = {Proceedings of the National Academy of Sciences},
	number = {35},
	pages = {12763--12768},
	publisher = {National Academy of Sciences},
	title = {Global alignment of multiple protein interaction networks with application to functional orthology detection},
	volume = {105},
	year = {2008}}

@article{sporns2005human,
	author = {Sporns, Olaf and Tononi, Giulio and K{\"o}tter, Rolf},
	journal = {PLoS computational biology},
	number = {4},
	pages = {e42},
	publisher = {Public Library of Science San Francisco, USA},
	title = {The human connectome: a structural description of the human brain},
	volume = {1},
	year = {2005}}

@article{vassaux2025,
  title={The feasibility of multi-graph alignment: a Bayesian approach},
  author={Vassaux, Louis and Massouli{\'e}, Laurent},
  journal={Advances in Applied Probability},
  pages={1--40},
  year={2025},
  publisher={Cambridge University Press}
}

@article{wu2022settling,
	author = {Wu, Yihong and Xu, Jiaming and Yu, Sophie H},
	journal = {IEEE Transactions on Information Theory},
	number = {8},
	pages = {5391--5417},
	publisher = {IEEE},
	title = {Settling the sharp reconstruction thresholds of random graph matching},
	volume = {68},
	year = {2022}}

@article{wu2023testing,
	author = {Wu, Yihong and Xu, Jiaming and Yu, Sophie H},
	journal = {The Annals of Applied Probability},
	number = {4},
	pages = {2519--2558},
	publisher = {Institute of Mathematical Statistics},
	title = {Testing correlation of unlabeled random graphs},
	volume = {33},
	year = {2023}}

@article{ochoa2026detection,
	author = {Ochoa, Daniel},
	note = {https://math.mit.edu/research/undergraduate/urop-plus/documents/2025/Ochoa.pdf},
	title = {A detection threshold for correlated multi-graphs},
	year = {2026}}

@article{feng2025strong,
	author = {Feng, Chenxu},
	journal = {arXiv preprint arXiv:2506.12752},
	title = {Strong Detection Threshold for Correlated {E}rd{\H{o}}s {R}\'enyi Graphs with Constant Average Degree},
	year = {2025}}

@article{otter1948number,
	author = {Otter, Richard},
	journal = {Annals of Mathematics},
	number = {3},
	pages = {583--599},
	publisher = {JSTOR},
	title = {The number of trees},
	volume = {49},
	year = {1948}}

@book{frieze2015introduction,
	author = {Frieze, Alan and Karo{\'n}ski, Micha{\l}},
	publisher = {Cambridge University Press},
	title = {Introduction to random graphs},
	year = {2015}}

%\appendix
\appendices

\section{Common tools for the asymmetric two-graph model}
\label{app:asym_common_tools}

This section presents the probabilistic tools used throughout the converse proofs for the asymmetric two-graph model. 

\subsection{Likelihood ratio}

Denote
\(
        t\coloneqq s_1s_2.
\)
The null law is denoted by $\Q$. For a fixed alignment $\pi\in\Sn$, let $\P_\pi$ denote the planted law conditioned on $\pistar=\pi$.  Thus
\(
    \P
    =
    \frac1{n!}\sum_{\pi\in\Sn}\P_\pi.
\)
%
%\subsection*{The likelihood ratio}
%
For a correctly aligned edge pair, the joint law under $\P_\pi$ is
\[
    \P_\pi\big((X^1_e,X^2_{\pi(e)})=(x_1,x_2)\big)
    =
    \begin{cases}
        1-p(s_1+s_2-s_1s_2), & (x_1,x_2)=(0,0),\\
        p(1-s_1)s_2, & (x_1,x_2)=(0,1),\\
        ps_1(1-s_2), & (x_1,x_2)=(1,0),\\
        ps_1s_2, & (x_1,x_2)=(1,1).
    \end{cases}
\]
Under $\Q$, the two coordinates are independent Bernoulli variables with parameters
$ps_1$ and $ps_2$.  Hence the single-edge likelihood ratio is
\begin{align}
\label{eq:asym_single_edge_lr}
    \ell(x_1,x_2)
    =
    \begin{cases}
        \frac{1-p(s_1+s_2-s_1s_2)}{(1-ps_1)(1-ps_2)},
            & (x_1,x_2)=(0,0),\\%[1.2em]
        \frac{1-s_1}{1-ps_1},
            & (x_1,x_2)=(0,1),\\%[1.2em]
        \frac{1-s_2}{1-ps_2},
            & (x_1,x_2)=(1,0),\\%[1.2em]
        \frac1p,
            & (x_1,x_2)=(1,1).
    \end{cases}
\end{align}
For a candidate alignment $\pi\in\Sn$, define
\[
        L_\pi(X^1,X^2)
        \coloneqq
        \prod_{e\in\En}\ell(X^1_e,X^2_{\pi(e)}),
        \qquad
        L(X^1,X^2)
        \coloneqq
        \frac1{n!}\sum_{\pi\in\Sn}L_\pi(X^1,X^2).
\]
Then $L=\dd\P/\dd\Q$ \footnote{We slightly abuse notation. Throughout, $\P$ denotes the joint distribution of $(X,\pi)$, where $X=(X^1,X^2)$. However, when we write $\frac{d\P}{d\Q}$ or $\TV(\P,\Q)$, the symbol $\P$ refers to $\P_X$, the marginal distribution of $X$ under $\P$.}. It is useful to record the centered single-edge expansion of $\ell$.  Define
\[
        u(x_1)
        \coloneqq
        \frac{x_1-ps_1}{\sqrt{ps_1(1-ps_1)}},
        ~~~~~~
        v(x_2)
        \coloneqq
        \frac{x_2-ps_2}{\sqrt{ps_2(1-ps_2)}}.
\]
Then
\begin{align}
\label{eq:asym_lr_orthogonal_expansion}
        \ell(x_1,x_2)
        =
        1+\rho_{\rm e} \, u(x_1)v(x_2),
        ~~~~ \mbox{ where }
        \rho_{\rm e}^2
        =
        \frac{(1-p)^2s_1s_2}{(1-ps_1)(1-ps_2)}.
\end{align}
In particular,
\(
    \rho_{\rm e}^2\leq s_1s_2=t.
\)

\subsection{Edge orbits and orbit-union graphs}

Let $K_n$ denote the complete graph on vertex set $[n]$.  Fix a true alignment
$\pistar\in\Sn$ and a candidate alignment $\pi\in\Sn$.  Define the relative permutation
\(
        \sigma\coloneqq {\pistar}^{-1}\circ\pi.
\)
The induced edge permutation (which we also denote by $\sigma$) partitions $\En$ into edge orbits.  We denote the collection of these orbits by
\[
        \mathcal O_\sigma
        \coloneqq
        \{O\subseteq\En: O \text{ is an orbit of } \sigma \}.
\]
We regard an orbit $O\subseteq\En$ as a graph whose edge set is $O$ and whose vertex set consists of the vertices incident to the edges of $O$.
An orbit $O\in\mathcal O_\sigma$ is called \emph{fully occupied} if
\(
        X^1_e=X^2_{\pistar(e)}=1
\)
for all
\(
    e\in O.
\)
Let $J_\sigma$ denote the graph obtained by taking the union of all fully occupied orbit
graphs.  Also define the true intersection graph
\[
        H^\star
        \coloneqq
        \{e\in\En: X^1_e=X^2_{\pistar(e)}=1\}.
\]
Under $\P_{\pistar}$, the graph $H^\star$ has distribution $\ER(n,pt)$, and $J_\sigma$ is
always a subgraph of $H^\star$.

\medskip 

\begin{lemma}
\label{lem:asym_orbit_calculation}
Fix $\pistar,\pi\in\Sn$, let $\sigma={\pistar}^{-1}\circ\pi$, and let
$O\in\mathcal O_\sigma$ be an edge orbit of length $r$.  Then
\begin{align}
\label{eq:asym_orbit_unconditional}
        \Expect_{\P_{\pistar}}
        \left[
            \prod_{e\in O}\ell(X^1_e,X^2_{\pi(e)})
        \right]
        =
        1+\rho_{\rm e}^{2r}.
\end{align}
\end{lemma}

\begin{proof}
Write the orbit as
\(
    O=(e_0,e_1,\ldots,e_{r-1}),
\)
where
\(
    e_{i+1}=\sigma(e_i)
\)
with indices understood modulo $r$.  Since $\pi={\pistar}\circ\sigma$, we have $\pi(e_i)={\pistar}(e_{i+1})$.  Let
\(
    A_i=X^1_{e_i}
\)
and
\(
    B_i=X^2_{{\pistar}(e_i)}.
\)
Under $\P_{\pistar}$, the pairs $(A_i,B_i)$ are independent over $i$.  Therefore, by changing
measure from $\P_{\pistar}$ to the corresponding product null law on these $r$ aligned pairs,
\begin{align*}
    \Expect_{\P_{\pistar}}
    \left[
        \prod_{i=0}^{r-1}\ell(A_i,B_{i+1})
    \right]
    =
    \Expect_{\Q}
    \left[
        \prod_{i=0}^{r-1}\ell(A_i,B_i)\ell(A_i,B_{i+1})
    \right].
\end{align*}
Using \eqref{eq:asym_lr_orthogonal_expansion}, this is
\[
    \Expect_\Q
    \bigg[ \, 
        \prod_{i=0}^{r-1}
        \Big(1+\rho_{\rm e}u(A_i)v(B_i)\Big) \cdot 
        \Big(1+\rho_{\rm e}u(A_i)v(B_{i+1})\Big)
    \bigg].
\]
Under $\Q$, all $u(A_i)$ and $v(B_i)$ have mean zero and variance one and are mutually
independent.  After expanding the product, a term has nonzero expectation only if every
variable $u(A_i)$ and $v(B_i)$ appears either zero times or two times.  On the cycle, the
only such choices are the empty choice and the choice of all $2r$ factors.  Thus the
expectation is $1+\rho_{\rm e}^{2r}$.
\end{proof}

For $\sigma\in\Sn$, define
\[
    \mathcal J_\sigma
    \coloneqq
    \Big\{
        J\subseteq K_n:
        J=\bigcup_{O\in\mathcal S} O
        \text{ for some }
        \mathcal S\subseteq\mathcal O_\sigma
    \Big\}.
\]
Here, as for the orbit graphs, the vertex set of a nonempty graph consists of the vertices incident to its edges. Thus $\mathcal J_\sigma$ is the collection of orbit-union graphs, and $J_\sigma\in\mathcal J_\sigma$ for every realization. Equivalently, for any subgraph $J$ of $K_n,$
\[    
        J\in\mathcal J_\sigma
        \qquad\Longleftrightarrow\qquad
        \sigma(J)=J,
\]
where $\sigma(J)$ denotes the graph whose edge set is $\{\sigma(e): e\in E(J)\}.$ 
Note that ${\cal J}_\sigma$ contains the null graph and graphs in  ${\cal J}_\sigma$ have no isolated vertices.
\medskip 

\begin{lemma}[Orbit-union estimates]
\label{lem:asym_dd_cond_orbit}
Fix ${\pistar},\sigma\in\Sn$ and let $J\in\mathcal J_\sigma$. Then
\begin{align}
\label{eq:asym_legal_contribution}
    \Expect_{\P_{\pistar}}
    \left[
        \ind{J_\sigma=J}
        L_{{\pistar}\circ\sigma}(X^1,X^2)
    \right]
    \leq
    t^{|E(J)|}.
\end{align}
Under the additional assumption $s_1 + s_2 \leq 1$ (or more generally if
\(
        (1-ps_1)(1-ps_2)\geq 1-p),
\)
for every edge orbit $O\in\mathcal O_\sigma$ of length $r$, if $F_O$
denotes the event that $O$ is fully occupied, then
\begin{align}
\label{eq:asym_conditional_orbit}
        \Expect_{\P_{\pistar}} \!
        \Big[
            \,
            \prod_{e\in O}\ell(X^1_e,X^2_{{\pistar}\circ\sigma(e)})
            \,\Big|\,
            F_O^c
        \Big]
        \leq 1,
\end{align}
whereas on $F_O$,
\(
        \prod_{e\in O}\ell(X^1_e,X^2_{{\pistar}\circ\sigma(e)})=p^{-r}.
\)
Consequently,
\begin{align}
\label{eq:asym_conditional_orbit_union}
    \Expect_{\P_{\pistar}}
    \left[
        L_{{\pistar}\circ\sigma}(X^1,X^2)
        \,\middle|\,
        J_\sigma=J
    \right]
    \leq
    p^{-|E(J)|}.
\end{align}
\end{lemma}

\begin{proof}
For $O\in\mathcal O_\sigma$, set
\[
    L_O
    \coloneqq
    \prod_{e\in O}\ell(X^1_e,X^2_{{\pistar}\circ\sigma(e)}).
\]
Then
\(
    L_{{\pistar}\circ\sigma}
    =
    \prod_{O\in\mathcal O_\sigma}L_O .
\)
Let $F_O$ be the event that $O$ is fully occupied.  Since $J \in \mathcal{J}_\sigma$, every orbit is
either contained in $J$ or disjoint from $J$, and hence
\[
        \{J_\sigma=J\}
        =
        \bigcap_{O\subseteq J}F_O
        \cap
        \bigcap_{O\not\subseteq J}F_O^c.
\]
Different edge orbits involve disjoint collections of aligned pairs
$(X^1_e,X^2_{{\pistar}(e)})$, so the corresponding factors are independent under $\P_{\pistar}$.
Thus  
\begin{align*}
    \Expect_{\P_{\pistar}}
    \left[
        \ind{J_\sigma=J}L_{{\pistar}\circ\sigma}
    \right]
    =
    \prod_{O\subseteq J}
    \Expect_{\P_{\pistar}}\big[\Indc_{F_O}L_O\big]
    \prod_{O\not\subseteq J}
    \Expect_{\P_{\pistar}}\big[\Indc_{F_O^c}L_O\big].
\end{align*}
If $|O|=r$, then on $F_O$ every factor in $L_O$ equals $\ell(1,1)=1/p$, while
$\P_{\pistar}(F_O)=(pt)^r$.  Therefore
\[
        \Expect_{\P_{\pistar}}\big[\Indc_{F_O}L_O\big]
        =
        (pt)^r p^{-r}
        =
        t^r.
\]
In contrast, by \Cref{lem:asym_orbit_calculation},
\(
        \Expect_{\P_{\pistar}}[L_O]=1+\rho_{\rm e}^{2r}.
\)
Using that $\rho_{\rm e}^2 \leq s_1 s_2$, we get
\[
        \Expect_{\P_{\pistar}}\big[\Indc_{F_O^c}L_O\big]
        =
        1+\rho_{\rm e}^{2r}-t^r
        \leq 1.
\]
Combining the last two displays gives
\[
    \Expect_{\P_{\pistar}}
    \left[
        \ind{J_\sigma=J}L_{{\pistar}\circ\sigma}
    \right]
    \leq
    \prod_{O\subseteq J}t^{|O|}
    =
    t^{|E(J)|} \, .
\]
For the conditional estimate, the assumption $(1-ps_1)(1-ps_2) \geq (1-p)$ gives
\[
        \rho_{\rm e}^2
        =
        \frac{(1-p)^2t}{(1-ps_1)(1-ps_2)}
        \leq
        (1-p)t.
\]
Hence
\(
        \rho_{\rm e}^{2r}
        \leq
        (1-p)^rt^r
        \leq
        (1-p^r)t^r
        =
        t^r-(pt)^r,
\)
and therefore
\[
        \Expect_{\P_{\pistar}}[L_O\mid F_O^c]
        =
        \frac{1+\rho_{\rm e}^{2r}-t^r}{1-(pt)^r}
        \leq 1.
\]
On $F_O$, every factor in $L_O$ equals $p^{-1}$, so
\(
        L_O=p^{-r}.
\)
The edge orbits are independent under $\P_{\pistar}$; thus, conditioning on
$J_\sigma=J$, the orbits contained in $J$ contribute $p^{-|E(J)|}$ and all
other orbits contribute at most one.  This proves
\eqref{eq:asym_conditional_orbit_union}.
\end{proof}

\subsection{Conditional second moment method}

The conditional second moment method has been useful for tightening converse bounds regarding the impossibility of weak detection (equivalent to $\TV(\P,\Q)\to 0.$)
See \cite[Section 4]{wu2023testing} and \cite{ding2023detection} for an introduction to the method.

\medskip

\begin{lemma}[Conditional second moment lemma]
\label{lem:conditional_second_moment}
Let $\calE$ be a set in the codomain of $(X,\pistar)$ (where $X=(X^1,X^2)$) and
suppose $\P\{(X,\pistar)\in\calE \}=1 - o(1)$ and one of the following two conditions is true:
\begin{align}
\Expect_{\Q}[(L')^2] & = 1 + o(1)   \label{eq:cond1_forTV} \\
\Expect_{\P}[\Indc_{(X,\pistar)\in \calE}L(X)] & = 1 + o(1)  \label{eq:cond2_forTV}
\end{align}
where $L' = L'(X)$ is the likelihood ratio between the $X$ marginal of $\P(~\cdot~ |(X,\pistar) \in \calE)$ and $\Q.$
Then $\TV(\P,\Q)\to 0.$  Equivalently, weak detection is impossible.
\end{lemma}
\begin{proof}
See \cite[Section 4]{wu2023testing} for a proof that the condition
 $\P\{(X,\pistar) \in \calE\}=1-o(1)$ and \eqref{eq:cond1_forTV}
 imply the conclusion.

Let $a = \P\{(X,\pistar) \in \calE\}.$  Suppose $a=1-o(1)$ and
 \eqref{eq:cond2_forTV} hold.
 Using
\begin{align}
\label{eq:asym_conditional_lr}
        L'(X)
        =
        \frac{1}{a}
\Expect_{\pi}[ 
            \Indc_{(X,\pi) \in \calE}
            L_\pi(X) ]
\end{align}
(where $\Expect_{\pi}$ denotes expectation with respect to a uniformly random $\pi\in\Sn$)
yields
\begin{align*}
        \Expect_\Q[(L')^2]
        &=
        \frac{1}{a^2} \, 
         \Expect_{\pi_1}\Expect_{ \pi_2} \left[ 
        \Expect_\Q
        \left[
            \Indc_{(X,\pi_1)\in \calE}
            \Indc_{(X,\pi_2)\in \calE}
            L_{\pi_1}(X)L_{\pi_2}(X)
        \right] \right]
        \nonumber\\
        & \leq
        \frac{1}{a^2} \, 
         \Expect_{\pi_1}\Expect_{ \pi_2} \left[ 
        \Expect_\Q
        \left[
            \Indc_{(X,\pi_1)\in \calE}
            L_{\pi_1}(X)L_{\pi_2}(X)
        \right] \right]
        \nonumber\\
        & =
        \frac{1}{a^2} \, 
         \Expect_{\pi_1} \left[ 
        \Expect_\Q
        \left[
            \Indc_{(X,\pi_1)\in \calE}
            L_{\pi_1}(X)L(X)
        \right] \right]
        \nonumber\\
                & =
        \frac{1}{a^2} \, 
         \Expect_{\pi_1} \left[ 
        \Expect_\P
        \left[
            \Indc_{(X,\pistar)\in \calE}
            L(X) \bigg| \pistar=\pi_1 ]
        \right] \right]
        \nonumber\\ 
        & =
        \frac{
            \Expect_\P[\Indc_{(X,\pistar) \in \mathcal E}L(X)]
        }{a^2}.
\end{align*}
Since $\Expect_{\Q}[(L')^2] \geq 1$ we conclude that the condition  $a=1-o(1)$ and \eqref{eq:cond2_forTV} together imply
\eqref{eq:cond1_forTV} and hence imply the conclusion of the lemma.
\end{proof}

\section{Converse proofs for detection in the asymmetric setting}

\label{app:detection_converses_asymmetric}

\subsection{Proof of~\texorpdfstring{\Cref{thm:asym_er_weak_detection_converses-i}}{}}
\label{app:asym_sparse_detection}

The proof follows that of~\cite{feng2025strong}, with the modifications needed to handle the asymmetric setting. A graph is a pseudoforest if every connected component contains at most one cycle; equivalently, every
connected component $C$ satisfies $|E(C)|\leq |V(C)|$.  Define the conditioning event
\[
        \mathcal E_{\rm pf}
        \coloneqq
        \{H^\star \text{ is a pseudoforest}\}.
\]

We use two elementary facts about pseudoforests.

\medskip

\begin{lemma}[Subcritical \erdosrenyi\ graphs are pseudoforests]
\label{lem:asym_er_pseudoforest}
Let \(0\leq c_n<1\) and let
\(
    G_n\sim\ER(n,c_n/n).
\)
Then
\(
    \P(G_n\text{ is not a pseudoforest})
    \leq \frac{2}{n(1-c_n)^3}.
\)
Hence, if \(c_n\leq 1-\omega(n^{-1/3})\), then
\(
    \P(G_n\text{ is a pseudoforest})\to1.
\)
Moreover, for every fixed \(\epsilon>0\), if \(c_n\leq1-\epsilon\),
then
\(
    \P(G_n\text{ is not a pseudoforest})
    =O(n^{-1}).
\)
\end{lemma}

\begin{proof}
We follow the enumeration in
\cite[Proof of Lemma~2.10, in particular (2.4)]
{frieze2015introduction}.
If \(G_n\) is not a pseudoforest, then it contains a minimal
connected subgraph containing two cycles. Such a subgraph on \(k\)
vertices consists of a path together with two additional edges, and
the number of such subgraphs on a fixed set of \(k\) labeled vertices
is at most \(k^2k!\). Therefore, by a union bound,
\begin{align*}
    \P(G_n\text{ is not a pseudoforest})
    &\leq
    \sum_{k=4}^n
        \binom{n}{k} k^2k!
        \left(\frac{c_n}{n}\right)^{k+1}  \\
    &\leq
    \frac{1}{n}\sum_{k=4}^\infty k^2c_n^{k+1}
    \leq
    \frac{2}{n(1-c_n)^3}.
\end{align*}
If \(c_n\leq1-\omega(n^{-1/3})\), then
\(n(1-c_n)^3\to\infty\), proving the first consequence.  If
\(c_n\leq1-\epsilon\), the bound is at most \(2/(n\epsilon^3)=O(n^{-1})\).
\end{proof}

\medskip

\begin{lemma}[Pseudoforest generating function]
\label{lem:asym_pseudoforest_generating}
Let $\mathcal G_{\rm pf}$ be the set of isomorphism classes of finite simple pseudoforests
with no isolated vertices, including the empty graph.  Define
\[
        G_{\rm pf}(t)
        \coloneqq
        \sum_{J\in\mathcal G_{\rm pf}}t^{|E(J)|}.
\]
There is a constant $\delta>0$ such that $G_{\rm pf}(t)<\infty$ for all
$0\leq t\leq\delta$.  Moreover,
\(
    G_{\rm pf}(t)=1+O(t)
\)
as
\(  
    t\to0.
\)
\end{lemma}

\begin{proof}
A connected component of a pseudoforest is either a tree or a connected unicyclic graph.
Let $a_k$ be the number of unlabeled trees on $k$ vertices, and let $b_k$ be the number of
unlabeled connected unicyclic graphs on $k$ vertices.  Since a pseudoforest is a multiset
of such components,
\begin{align}
\label{eq:asym_gpf_product}
        G_{\rm pf}(t)
        =
        \prod_{k\geq2}(1-t^{k-1})^{-a_k}
        \prod_{k\geq2}(1-t^k)^{-b_k}.
\end{align}
By Otter's theorem \cite{otter1948number}, there is a constant $C<\infty$ and a constant
$\alpha>0$ such that
\(
        a_k\leq C\alpha^{-k}k^{-3/2}.
\)
Deleting one edge from the cycle of a connected unicyclic graph gives a tree, and after choosing the
tree there are at most $k^2$ possible edges to add back.  Hence
\(
    b_k\leq k^2a_k\leq C\alpha^{-k}k^{1/2}.
\)
Choose $\delta>0$ sufficiently small.  For $0\leq t\leq\delta$, using
$-\log(1-x)\leq2x$ for small $x$ gives
\[
        \log G_{\rm pf}(t)
        \leq
        2\sum_{k\geq2}a_kt^{k-1}
        +
        2\sum_{k\geq2}b_kt^k
        =
        O(t).
\]
Therefore $G_{\rm pf}(t)<\infty$ for $t\leq\delta$ and
$G_{\rm pf}(t)=e^{O(t)}=1+O(t)$ as $t\to0$.
\end{proof}

\medskip 

\begin{proposition}[Pseudoforest-conditional likelihood ratio]
\label{prop:asym_conditional_lr}
The conditional likelihood ratio satisfies
\[
        \Expect_\P\big[\Indc_{\mathcal E_{\rm pf}}L(X^1,X^2)\big]
        \leq
        G_{\rm pf}(t).
\]
In particular, if $t\to0$, then
\(
        \Expect_\P\big[\Indc_{\mathcal E_{\rm pf}}L(X^1,X^2)\big]
        \leq
        1+o(1).
\)
\end{proposition}

\begin{proof}
Condition on the true alignment $\pistar$.  Since
\(
    L(X^1,X^2)
    =
    \frac1{n!}\sum_{\pi\in\Sn}L_\pi(X^1,X^2),
\)
and since $\pi={\pistar}\circ\sigma$ ranges over $\Sn$ as $\sigma$ ranges over $\Sn$, we get
\[
    \Expect_\P\big[\Indc_{\mathcal E_{\rm pf}}L(X^1,X^2)\big]
    =
    \frac1{n!}\sum_{\sigma\in\Sn}
    \Expect_{\P_{\pistar}}
    \big[
        \Indc_{\mathcal E_{\rm pf}}L_{{\pistar}\circ\sigma}(X^1,X^2)
    \big].
\]
For fixed $\sigma$, split according to the value
$J_\sigma\in\mathcal J_\sigma$.  On
$\mathcal E_{\rm pf}$, the graph $J_\sigma$ is a subgraph of the pseudoforest $H^\star$, and is therefore itself a pseudoforest.  Applying
\Cref{lem:asym_dd_cond_orbit} gives
\[
    \Expect_\P\big[\Indc_{\mathcal E_{\rm pf}}L(X^1,X^2)\big]
    \leq
    \frac1{n!}
    \sum_{\sigma\in\Sn}
    \sum_{\substack{J \in \mathcal{J}_\sigma\\
                    J\text{ is a pseudoforest}}}
        t^{|E(J)|}.
\]
Since
\(
    J\in\mathcal J_\sigma
\)
is equivalent to
\(
    \sigma(J)=J,
\)
we may switch the order of summation:
\[
    \frac1{n!}
    \sum_{\substack{J\subseteq K_n\\ J\text{ is a pseudoforest} }}
        t^{|E(J)|}
        \big|\{\sigma\in\Sn:\sigma(J)=J\}\big|.
\]
%Here each graph is regarded as having no isolated vertices.
If $J$ has $v$ vertices, then
\(
    \big|\{\sigma\in\Sn:\sigma(J)=J\}\big|
    =
    (n-v)!|\operatorname{Aut}(J)|.
\)
For an isomorphism class $\mathbf J$ with $v$ vertices, the number of its labeled copies in $K_n$ is
\[
        \operatorname{Sub}_n(\mathbf J)
        =
        \frac{n!}{(n-v)!|\operatorname{Aut}(\mathbf J)|}.
\]
Thus, after summing over all labeled copies of a fixed isomorphism class, the factors
cancel and leave only $t^{|E(\mathbf J)|}$.  Therefore
\[
    \Expect_\P\big[\Indc_{\mathcal E_{\rm pf}}L(X^1,X^2)\big]
    \leq
    \sum_{\substack{\mathbf J\in\mathcal G_{\rm pf}:\\ |V(\mathbf J)|\leq n}}
        t^{|E(\mathbf J)|}
    \leq
    G_{\rm pf}(t).
\]
The final claim follows from \Cref{lem:asym_pseudoforest_generating}.
\end{proof}

We now prove \Cref{thm:asym_er_weak_detection_converses-i}.  Let
\(
        c_n\coloneqq npt.
\)
By assumption,
\(
        c_n\leq 1-\omega(n^{-1/3}).
\)
Since $H^\star\sim\ER(n,pt)=\ER(n,c_n/n)$ under the planted law, \Cref{lem:asym_er_pseudoforest} gives
\(
    \P(\mathcal E_{\rm pf})\to1.
\)
This together with the last line of \Cref{prop:asym_conditional_lr} and \Cref{lem:conditional_second_moment} completes the proof of
\Cref{thm:asym_er_weak_detection_converses-i}.

\subsection{Proof of~\texorpdfstring{\Cref{thm:asym_er_weak_detection_converses-ii}}{}}
\label{app:asym_dd_detection}

We use the notation from Appendix~\ref{app:asym_common_tools}; in particular,
\(
        t=s_1s_2,
\)
and
\(
        L=\frac{\dd\P}{\dd\Q}
        =
        \frac1{n!}\sum_{\pi\in\Sn}L_\pi,
\)
and the single-edge likelihood ratio $\ell$ is given by
\eqref{eq:asym_single_edge_lr}.
Set
\(
    \lambda_n\coloneqq npt
\)
and
\( 
    \lambda_\alpha \coloneqq \varrho^{-1}(1/\alpha).
\)
Assume
\[
        p=n^{-\alpha+o(1)},
        ~~~~~~
        \alpha\in(0,1),
        ~~~~~~
    \max\{s_1,s_2\} = 1 -\Omega(1),
        ~~~~~~
        \lambda_n\leq \lambda_\alpha-\epsilon .
\]
Recall that under $\P$, the true intersection graph is 
\(
        H^\star\sim\ER(n,pt)=\ER(n,\lambda_n/n).
\)

\medskip 
 
We now adapt the conditional second moment argument of~\cite{ding2023detection}.  Choose a constant
$\overline\lambda<\lambda_\alpha$ such that $\lambda_n\leq\overline\lambda$
for all sufficiently large $n$, and define
\[
        \xi\coloneqq
        \frac{\varrho(\overline\lambda)+1/\alpha}{2}
        <\frac1\alpha .
\]
Following ~\cite[Section III.B]{ding2023detection}, define a graph $H$ to be admissible if: (i) for every $U\subseteq [n]$, $|E_H(U)| \leq \xi |U|$; (ii) the maximum degree of $H$ is less than $\log n$; (iii)  every connected subgraph of size less than $\log\log n$ contains at most one cycle; and (iv) for any $k \geq 2$, the number of $k$-cycles is bounded by $n^{\delta k}$, where $\delta < \delta_0$ and $n p^{\xi} \geq n^{\delta_0}$ for some $\delta_0 > 0$. 

Let $\mathcal E_{\rm adm}$ be the event that $H^\star$ is admissible.  Since $H^\star\sim\ER(n,\lambda_n/n)$ and $\lambda_n\leq\overline\lambda$, the proof of \cite[Lemma III.4]{ding2023detection} gives
\(
        \P(\mathcal E_{\rm adm})=1-o(1).
\)
Moreover, the conditional lower bound in \cite[Lemma III.4]{ding2023detection} also applies verbatim (it is about a deterministic property of a realization of the intersection graph and not about how the graph is generated) so there is a constant $c_0>0$ such that, uniformly over every relative permutation $\sigma$ and every $J\in\mathcal J_\sigma$ satisfying the admissibility conditions above,
\[
        \P(\mathcal E_{\rm adm}\mid \pistar,J_\sigma=J)
        \geq
        (1-o(1)) \, c_0^{|E(J)|}.
\]
The restriction on $J$ is important: if $J$ already violates admissibility,
the conditional probability on the left is zero.  Under the planted law
conditioned on $\mathcal E_{\rm adm}$, however, $J_\sigma\subseteq H^\star$, and
each of the four admissibility conditions is inherited by subgraphs.  Thus
every value of $J_\sigma$ that appears below satisfies this restriction.

Let $\P'$ be the marginal law of the observed pair under the planted law
conditioned on $\mathcal E_{\rm adm}$, and let
\(
        L'
        \coloneqq
        {\dd\P'}/{\dd\Q}.
\)
By \Cref{lem:conditional_second_moment}, it is enough to show
\(
        \Expect_\Q[(L')^2]=1+o(1),
\)
because $\P(\mathcal E_{\rm adm})=1-o(1)$.
\Cref{lem:asym_dd_cond_orbit} applies because the assumption $\max\{s_1,s_2\} = 1-\Omega(1)$ and the fact $s_1s_2\to 0$ implies $s_1 + s_2 \leq 1$ for large enough $n,$  so the conclusion of \cite[Lemma III.1]{ding2023detection},
namely \eqref{eq:asym_conditional_orbit_union}, holds.  
Therefore \cite[Lemma III.3]{ding2023detection} goes through as in the symmetric case with the same proof, yielding:
\[
        \Expect_\Q[(L')^2]
        \leq
        \frac{1+o(1)}{\P(\mathcal E_{\rm adm})} \,
        \Expect_{\P'} \!
        \left[
            \frac1{n!}
            \sum_{\sigma\in\Sn}
            (c_0p)^{-|E(J_\sigma)|}
        \right].
\]

It remains to bound the last sum.  This is identical to the deterministic enumeration in \cite[Proposition III.7]{ding2023detection}.  That enumeration uses only the fact that the realized graph $H^\star$ is admissible and the condition
\(
        np^\xi\geq n^{\Omega(1)},
\)
which follows from $\xi<1/\alpha$ and $p=n^{-\alpha+o(1)}$.  It does not use $s_1$ and $s_2$ separately. Hence, whenever $H^\star$ is admissible,
\[
        \frac1{n!}
        \sum_{\sigma\in\Sn}
        (c_0p)^{-|E(J_\sigma)|}
        \leq
        1+o(1).
\]
Therefore
\(
        \Expect_\Q[(L')^2]\leq 1+o(1),
\)
completing the proof of  \Cref{thm:asym_er_weak_detection_converses-ii}.

\subsection{Proof of~\texorpdfstring{\Cref{thm:asym_er_weak_detection_converses-iii}}{}}
\label{app:asym_dense_detection}

Throughout this proof, set
\(       
    A\coloneqq X^1
\)
and 
\( 
    B\coloneqq X^2,
\)
and denote
\(
    t\coloneqq s_1s_2
\)
and
\(
        I_p\coloneqq p\left(\log\frac1p-1+p\right).
\)
Without loss of generality, assume \eqref{eq:lb-denseER} holds with
equality.  Indeed, if the left-hand side of \eqref{eq:lb-denseER} is
smaller, one can increase the subsampling probabilities, and independently thin the observed graphs back to the original values; total variation cannot increase under this processing.  Thus, we assume
\begin{align}
        ntI_p=(2-\epsilon)\log n.
        \label{eq:asym_dense_sharp_lower_bound}
\end{align}
In the dense graph regime,
$p \le 1- \Omega(1)$  and  $p=n^{-o(1)}$, 
\eqref{eq:asym_dense_sharp_lower_bound} implies
\begin{align}
        npt=\omega(1),
        \qquad
        t=n^{-1+o(1)}.
        \label{eq:asym_dense_npt}
\end{align}

For fixed $\pi,\widetilde\pi\in\Sn$, define
\(
        \sigma\coloneqq \pi^{-1}\circ\widetilde\pi.
\)
Let $n_k$ (resp. $N_k$) denote the number of $k$-cycles of $\sigma$ (resp. $\sigma^{\rm E}$, the induced edge permutation of $\sigma$).  In particular,
\(
        N_1=\binom{n_1}{2}+n_2.
\)
For $i<j$, define
\begin{align}
        X_{ij}
        \coloneqq
        \ell(A_{ij},B_{\pi(i)\pi(j)})
        \cdot \ell(A_{ij},B_{\widetilde\pi(i)\widetilde\pi(j)}),
        \label{eq:asym_dense_Xij}.
\end{align}
where  $\ell$ is the single-edge likelihood ratio given in
\eqref{eq:asym_single_edge_lr}.

Let $\mathcal O$ be the collection of edge orbits of the induced edge
permutation, and for $O\in\mathcal O$ set
\(
        X_O\coloneqq \prod_{\{i,j\}\in O}X_{ij}.
\)
As in the orbit calculation of \Cref{lem:asym_orbit_calculation}, applied
after changing measure by $L_\pi$, for every edge orbit $O$,
\begin{align}
        \Expect_{\Q}[X_O]
        =
        1+\rho_{\rm e}^{2|O|}.
        \label{eq:asym_dense_orbit_expectation}
\end{align}

For $S\subset[n]$ and a candidate alignment $\pi$, define
\[
        e_A(S)\coloneqq \sum_{\{i,j\}\subset S}A_{ij},
        \qquad
        e_{B^\pi}(S)\coloneqq
        \sum_{\{i,j\}\subset S}B_{\pi(i)\pi(j)},
        \qquad
        e_{A\wedge B^\pi}(S)
        \coloneqq
        \sum_{\{i,j\}\subset S}A_{ij}B_{\pi(i)\pi(j)}.
\]

By \eqref{eq:asym_single_edge_lr}, it follows that 
\(
        \ell(a,b)=C_{00} \, C_{10}^{a} \, C_{01}^{b} \, C_{11}^{ab} \, ,
\)
where
\begin{align}
        C_{00}
        &=
        1+\frac{p(1-p)s_1s_2}{(1-ps_1)(1-ps_2)}
        =
        \exp\{(1+o(1)) \cdot p(1-p)s_1s_2\},
        \label{eq:C00_bnd} \\
        C_{10}
        &=
        1-\frac{s_2(1-p)}{1-p(s_1+s_2-s_1s_2)}
        =
        \exp\{-(1+o(1)) \cdot s_2(1-p)\},
       \label{eq:C10_bnd}\\
        C_{01}
        &=
        1-\frac{s_1(1-p)}{1-p(s_1+s_2-s_1s_2)}
        =
        \exp\{-(1+o(1)) \cdot s_1(1-p)\},
        \label{eq:C01_bnd}\\
        C_{11}
        &%= \frac1p\frac1{C_{00}C_{10}C_{01}}
        =\frac{1-p(s_1+s_2-s_1s_2)}{p(1-s_1)(1-s_2)}
        =
        (1+o(1)) \cdot \frac1p.
        \label{eq:C11_bnd}
\end{align}
The second equality in \eqref{eq:C11_bnd} uses the assumption $\max\{s_1,s_2\}\to 0.$
Let $F$ denote the set of fixed points of
$\sigma=\pi^{-1}\circ\widetilde\pi$, and let $\mathcal O_1$ denote the
collection of length-one edge orbits inside $F$.  Then by
\eqref{eq:asym_dense_Xij},
\begin{align}
        \prod_{O\in\mathcal O_1}X_O
        &=
        \prod_{\{i,j\}\subset F}X_{ij}
        =
        C_{00}^{2\binom{n_1}{2}}
        C_{10}^{2e_A(F)}
        C_{01}^{2e_{B^\pi}(F)}
        C_{11}^{2e_{A\wedge B^\pi}(F)}.
        \label{eq:asym_dense_X_OF_ER}
\end{align}
Note that $C_{11}$ can be much larger than $C_{00},C_{10}$ and $C_{01}$.
Thus, when $e_{A\wedge B^\pi}(F)$ is atypically large,
$\prod_{\{i,j\}\subset F}X_{ij}$ becomes enormously large, driving the
unconditional second moment to explode.  Hence, we truncate
$\prod_{\{i,j\}\subset F}X_{ij}$ by conditioning on the maximum possible
value of $e_{A\wedge B^\pi}(F)$ under the planted model when $|F|=n_1$ is
large.

Specifically, for $2\le k\le n$, define
\begin{align}
        \zeta(k)
        \coloneqq
        \binom{k}{2}pt
        \exp\left\{
            1+
            W\left(
                \frac{2\log(2en/k)}{e(k-1)pt}
                -
                \frac1e
            \right)
        \right\},
        \label{eq:asym_dense_zeta_def}
\end{align}
where $W$ is the principal branch of the Lambert $W$ function.   Define the set
\begin{align}
        \calE
        \coloneqq
        \bigcap_{S\subset[n]:\, I_pn\le |S|\le n}\calE_S,
        \label{eq:asym_dense_calE}
\end{align}
where for each $S\subset[n]$,
\begin{align}
\calE_S& \triangleq \bigg\{  (A, B, \pi):  e_A(S) \ge \binom{|S| }{2} ps_1 - \sqrt{2 \binom{|S|}{2}  ps_1 |S| \log \frac{2en}{ |S|}  },\nonumber \\
&\hspace{1cm} e_B(S) \ge \binom{|S| }{2} ps_2 - \sqrt{2 \binom{|S|}{2}  ps_2 |S| \log \frac{2en}{ |S|}  },~~~
 e_{A\wedge B^{\pi}}(S) \le   \zeta \left( |S| \right) \bigg \}
 \label{eq:cond_high_prob_dense_ER}.
\end{align}
We shall consider $\P$ conditioned on the event $\{(A,B,\pistar)\in \calE\}.$

\medskip 

\begin{lemma}
\label{lem:asym_dense_cond_high_prob}
Under the planted law,
\(
        \P\big((A,B,\pistar)\in \calE\big)
        =
        1-e^{-\Omega(I_pn)}
        =
        1-o(1).
\)
\end{lemma}

\begin{proof}
Fix an integer $I_pn\le k\le n$ and let $M=\binom{k}{2}$.  Let
\[
        t_1=\sqrt{2Mps_1\log(1/\delta)},
        \qquad
        t_2=\sqrt{2Mps_2\log(1/\delta)},
        \qquad
        t'
        =
        Mpt
        \exp\left\{
            1+
            W\left(
                \frac{\log(1/\delta)}{eMpt}
                -
                \frac1e
            \right)
        \right\},
\]
for a parameter $\delta$ to be specified later.  Fix a subset
$S\subset[n]$ with $|S|=k$.  Under the planted law, conditioned on
$\pistar$,
\[
        e_A(S)\sim\Bin(M,ps_1),
        \qquad
        e_{B^{\pistar}}(S)\sim\Bin(M,ps_2),
        \qquad
        e_{A\wedge B^{\pistar}}(S)\sim\Bin(M,pt).
\]
By the Chernoff lower-tail bound, with probability at least $1-2\delta$, it holds that
\(
        e_A(S)\ge Mps_1-t_1,
\)
and
\(
        e_{B^{\pistar}}(S)\ge Mps_2-t_2.
\)
Moreover, the Chernoff bound also yields, with probability at least $1-\delta$,
\(
        e_{A\wedge B^{\pistar}}(S)\le t'.
\)

There are $\binom nk \leq (en/k)^k$ subsets $S$ with $|S|=k$.  Choose
\(
        1/\delta=\left({2en}/{k}\right)^k.
\)
Then $t'=\zeta(k)$, and a union bound over all $S$ of size $k$ gives failure probability at most $3\cdot 2^{-k}$.  Summing over
$k\ge I_pn$ gives
\[
        \P\big((A,B,\pistar)\not\in \calE\big)
        \le
        3\sum_{k\ge I_pn}2^{-k}
        =
        e^{-\Omega(I_pn)}.
\]
Finally, since $p=n^{-o(1)}$ and $p\le 1-\Omega(1)$, we have
$I_p=n^{-o(1)}$ and hence $I_pn=n^{1-o(1)}$.
\end{proof}

Let $\P'$ denote the planted law conditioned on $\{(A,B,\pistar)\in \calE\}$ and let
\(
        L'\coloneqq \frac{\dd\P'}{\dd\Q}.
\)
By  \Cref{lem:conditional_second_moment}, it is enough to show
\(
        \Expect_\Q[(L')^2]\le 1+o(1).
\)
By \Cref{lem:asym_dense_cond_high_prob},
\begin{align}
        \Expect_\Q[(L')^2]
        &=
        (1+o(1)) \cdot 
        \Expect_{\pi\ci\widetilde\pi} \
        \Expect_\Q
        \Big[
            \prod_{O\in\mathcal O}X_O
            \ind{(A,B,\pi)\in \calE}
            \ind{(A,B,\widetilde{\pi})\in \calE}
        \Big].
        \label{eq:asym_dense_cond_second_moment}
\end{align}

We now fix $\pi,\widetilde\pi$ and separately consider two cases.

\noindent
{\bf Case 1: $n_1\le I_p n$.}
In this case, we drop the indicators and use the unconditional second moment:
\begin{align}
&\Expect_\Q
\Big[
    \prod_{O\in\mathcal O}X_O
            \ind{(A,B,\pi)\in \calE}
            \ind{(A,B,\widetilde{\pi})\in \calE}
\Big]
\le
\Expect_\Q
\Big[
    \prod_{O\in\mathcal O}X_O
\Big]
=
\prod_{O\in\mathcal O}
\left(1+\rho_{\rm e}^{2|O|}\right),
\label{eq:asym_dense_case1}
\end{align}
where the equality follows from \eqref{eq:asym_dense_orbit_expectation}.

\noindent
{\bf Case 2: $n_1>I_p n$.}
In this case,
\begin{align}
\Expect_\Q
\Big[
    \prod_{O\in\mathcal O}X_O
            \ind{(A,B,\pi)\in \calE}
            \ind{(A,B,\widetilde{\pi})\in \calE}
\Big]
& \le
\Expect_\Q
\Big[
    \prod_{O\in\mathcal O}X_O
    \ind{(A,B,\pi)\in \calE_F}
\Big] \nonumber \\
& =
\prod_{O\notin\mathcal O_1}
\Expect_\Q[X_O]\,
\Expect_\Q
\Big[
    \prod_{O\in\mathcal O_1}X_O
    \ind{(A,B,\pi)\in \calE_F}
\Big]
\nonumber\\
&=
\prod_{O\notin\mathcal O_1}
\left(1+\rho_{\rm e}^{2|O|}\right) \cdot 
\Expect_\Q
\Big[
    \prod_{\{i,j\}\subset F}X_{ij}
    \ind{(A,B,\pi)\in \calE_F}
\Big].
\label{eq:asym_dense_case2_factor}
\end{align}
Here $\calE_F$ denotes the event $\calE_S$   with $S=F$.  The inequality follows from $\calE\subseteq\calE_F$ when $|F|=n_1>I_pn$.  The factorization follows because $X_O$ is a function of the variables $(A_{ij},B_{\pi(i)\pi(j)})_{\{i,j\}\in O}$; these variables are independent across distinct edge orbits under $\Q$, and the event ${(A,B,\pi)\in \calE_F}$ depends only on the variables belonging to the length-one edge orbits inside $F$.

Let $M=\binom{n_1}{2}$.  Under $(A,B,\pi) \in \calE_F$, we have
\begin{align} \label{eq:Mps}
        e_A(F)\ge (1+o(1))Mps_1
        \hspace{0.5cm}  \mbox{and} \hspace{0.5cm}
        e_{B^\pi}(F)\ge (1+o(1))Mps_2.
\end{align}
Indeed, for $i\in\{1,2\}$,
\begin{align}
        \frac{n_1\log(n/n_1)}{Mps_i}
        &=
        \frac{2\log(n/n_1)}{(n_1-1)ps_i}
        =
        O\left(
            \frac{\log(1/I_p)}{I_pnps_i}
        \right)
        =
        \Theta\left(
            \frac1{np^2s_i}
        \right)
        =
        o(1),
        \label{eq:asym_dense_edge_count_concentration}
\end{align}
where we used $n_1\ge I_pn$, the relation
$(1/I_p)\log(1/I_p)=\Theta(1/p)$, and $np^2s_i=\omega(1).$   Moreover, under $\calE_F$,
\(
        e_{A\wedge B^\pi}(F)\le \zeta(n_1).
\)

Let
\begin{align}
        \gamma\equiv \gamma(n_1)
        \coloneqq
        \frac{2\log(2en/n_1)}{(n_1-1)pt}.
        \label{eq:asym_dense_gamma}
\end{align}
Then
\(
        \zeta(n_1)
        =
        Mpt
        \exp\big\{
            1+
            W\left({\gamma-1}/{e}\right)
        \big\}.
\)
The following estimate is the Lambert-$W$ estimate used in~\cite{wu2023testing}, with $s^2$ replaced by $t=s_1s_2$.

\medskip 

\begin{lemma}
\label{lem:asym_dense_W_function}
For $n_1\ge I_pn$, the following hold.
\begin{itemize}
\item If $\gamma=o(1)$, then
\(
        \zeta(n_1)=(1+o(1))Mpt.
\)
\item If $\gamma=\Theta(1)$, then
\(
        \zeta(n_1)=\Theta(Mpt).
\)
In particular, uniformly for $n_1\ge I_pn$,
\(
        \zeta(n_1)=o(Mt).
\)
\item If $\gamma=\omega(1)$, then
\(
        \zeta(n_1)
        \le
        (e+o(1))Mpt\,\frac{\gamma}{\log\gamma}.
\)
In particular, uniformly for $n_1\ge I_pn$,
\(
        \zeta(n_1)=o(Mt).
\)
\end{itemize}
\end{lemma}

\begin{proof}
This is \cite[Lemma 7]{wu2023testing}, with the substitution
$s^2\mapsto t=s_1s_2$.  The proof only uses the definition
\eqref{eq:asym_dense_zeta_def}, the dense-regime condition
$p=n^{-o(1)}$, and the boundary relation
\eqref{eq:asym_dense_sharp_lower_bound}.
\end{proof}

 In view of \Cref{lem:asym_dense_W_function}, we get that
\(
 \zeta(n_1)\le Mpt+o(Mt)
\)
for all 
\(
n_1\ge I_pn.
\)
For ease of notation, we henceforth write $\zeta(n_1)$ simply as $\zeta$.
It follows from \eqref{eq:asym_dense_X_OF_ER} that
\begin{align}
    \Expect_\Q
    \Big[
        \prod_{\{i,j\}\subset F}X_{ij}
        \ind{(A,B,\pi)\in \calE_F}
    \Big]
    &=
    C_{00}^{2M}
    \Expect_\Q
    \left[
        C_{10}^{2e_A(F)}
        C_{01}^{2e_{B^\pi}(F)}
        C_{11}^{2e_{A\wedge B^\pi}(F)}
        \ind{(A,B,\pi)\in \calE_F}
    \right]
    \nonumber\\
    &\le
    C_{00}^{2M}
    C_{10}^{(2+o(1))Mps_1}
    C_{01}^{(2+o(1))Mps_2}
    \Expect_\Q
    \left[
        C_{11}^{2e_{A\wedge B^\pi}(F)}
        \ind{e_{A\wedge B^\pi}(F)\le \zeta}
    \right]
    \nonumber\\
    &=
    \exp\{-(2+o(1))Mpt(1-p)\}
    \Expect_\Q
    \left[
        C_{11}^{2e_{A\wedge B^\pi}(F)}
        \ind{e_{A\wedge B^\pi}(F)\le \zeta}
    \right],
    \label{eq:asym_dense_fixed_point_pre_mgf}
\end{align}
where the last equality uses \eqref{eq:C00_bnd}-\eqref{eq:C01_bnd} and \eqref{eq:Mps}.
Let
\[
        u\coloneqq C_{11}^2=(1+o(1))p^{-2}.
\]
Under $\Q$, we have
\(
        e_{A\wedge B^\pi}(F)\sim\Bin(M,p^2t).
\)
Then, for any $\lambda\in[0,1]$,
\begin{align}
    &\Expect_\Q
    \left[
        C_{11}^{2e_{A\wedge B^\pi}(F)}
        \ind{e_{A\wedge B^\pi}(F)\le \zeta}
    \right]
    \le
    \Expect_\Q
    \left[
        u^{\lambda e_{A\wedge B^\pi}(F)+(1-\lambda)\zeta}
    \right]
    =
    u^{(1-\lambda)\zeta}
    \left(1+p^2t(u^\lambda-1)\right)^M.
    \label{eq:asym_dense_holder}
\end{align}
Optimizing over $\lambda\in[0,1]$, equivalently over $y=u^\lambda\in[1,u]$, gives
\begin{align}
    \Expect_\Q
    \left[
        C_{11}^{2e_{A\wedge B^\pi}(F)}
        \ind{e_{A\wedge B^\pi}(F)\le \zeta}
    \right]
    \le
    \left(
        \frac{M(1-p^2t)}{M-\zeta}
    \right)^M
    \left(
        \frac{\zeta(1-p^2t)}{u p^2t(M-\zeta)}
    \right)^{-\zeta}.
    \label{eq:asym_dense_mgf_optimized}
\end{align}
Combining
\eqref{eq:asym_dense_fixed_point_pre_mgf} and
\eqref{eq:asym_dense_mgf_optimized}, we obtain
\begin{align}
    \Expect_\Q
    \Big[
        \prod_{\{i,j\}\subset F}X_{ij}
        \ind{(A,B,\pi)\in \calE_F}
    \Big]
    &\le
    \exp\left\{
        -(2+o(1))Mpt(1-p)
        +
        M\log\frac{M(1-p^2t)}{M-\zeta}
        +
        \zeta\log
        \frac{(M-\zeta)up^2t}{\zeta(1-p^2t)}
    \right\}
    \nonumber\\
    &\le
    \exp\left\{
        -Mpt(2-p)
        +
        \zeta\log\frac{eMt}{\zeta}
        +
        o(\zeta)
    \right\}.
    \label{eq:asym_dense_fixed_point_bound}
\end{align}
The last inequality is the same simplification as in~\cite[Appendix A.3]{wu2023testing}: it uses $\zeta=Mt(p+o(1))$, $u=(1+o(1))p^{-2}$, $(M-\zeta)\log(1-p^2t)=-(1+o(1))Mp^2t$, $\log(up^2)=o(1)$, and $h(x)=x\log(e/x)+o(x)$ as $x=o(1)$.

Combining the two cases yields
\begin{align*}
        \Expect_\Q[(L')^2]
        & \le
        (1+o(1)) \cdot 
        \Expect_\sigma
        \left[
            \prod_{O\in\mathcal O}
            \left(1+\rho_{\rm e}^{2|O|}\right)
            \ind{n_1\le I_pn}
        \right]
        \\
        & +
        (1+o(1)) \cdot 
        \Expect_\sigma
        \Big[
            \prod_{O\notin\mathcal O_1}
            \left(1+\rho_{\rm e}^{2|O|}\right)
            \exp\left\{
                -Mpt(2-p)
                +
                \zeta\log\frac{eMt}{\zeta}
                +
                o(\zeta)
            \right\}
            \ind{n_1>I_pn}
        \Big],
\end{align*}
where $\sigma$ is uniform on $\Sn$ and $M=\binom{n_1}{2}$ in the second
term.

Recall that $\rho_{\rm e}^2\le t$.  Since $t=n^{-1+o(1)}$,
\[
        \prod_{k\ge3}(1+\rho_{\rm e}^{2k})^{N_k}
        \le
        \exp\left\{\rho_{\rm e}^6\sum_{k\ge3}N_k\right\}
        \le
        \exp\{n^2\rho_{\rm e}^6/2\}
        =
        1+o(1).
\]
Thus
\begin{align}
        \prod_{O\notin\mathcal O_1}
        \left(1+\rho_{\rm e}^{2|O|}\right)
        &=
        (1+o(1))(1+\rho_{\rm e}^2)^{n_2}
        (1+\rho_{\rm e}^4)^{N_2}
        \le
        (1+o(1))
        \exp\{tn_2+t^2N_2\}.
        \label{eq:asym_dense_orbits_not_fixed}
\end{align}
Similarly,
\[
        \prod_{O\in\mathcal O_1}
        \left(1+\rho_{\rm e}^{2|O|}\right)
        =
        (1+\rho_{\rm e}^2)^{\binom{n_1}{2}}
        \le
        \exp\{ tn_1^2/2 \} \, .
\]
Therefore,
\begin{align}
        \Expect_\Q[(L')^2]
        &\le
        (1+o(1))
        \Expect_\sigma
        \left[
            \exp\left\{
                \frac{tn_1^2}{2}
                +
                tn_2
                +
                t^2N_2
            \right\}
            \ind{n_1\le I_pn}
        \right]
        \nonumber\\
        &\quad+
        (1+o(1))
        \Expect_\sigma
        \left[
            \exp\left\{
                tn_2
                +
                t^2N_2
                -
                Mpt(2-p)
                +
                \zeta\log\frac{eMt}{\zeta}
                +
                o(\zeta)
            \right\}
            \ind{n_1>I_pn}
        \right].
        \label{eq:asym_dense_two_terms}
\end{align}

The remaining summation over the random permutation is identical to the corresponding part of~\cite{wu2023testing}, after replacing $s^2$ by $t=s_1s_2$.

\medskip

\begin{lemma}
\label{lem:asym_dense_wxy_tail}
Assuming $p \le 1- \Omega(1)$  and  $p=n^{-o(1)}$ and
\eqref{eq:asym_dense_sharp_lower_bound},
\begin{align}
        \Expect_\sigma
        \left[
            \exp\left\{
                \frac{tn_1^2}{2}
                +
                tn_2
                +
                t^2N_2
            \right\}
            \ind{n_1\le I_pn}
        \right]
        &\le
        1+o(1),
        \label{eq:asym_dense_wxy_tail_first}
        \\
        \Expect_\sigma
        \left[
            \exp\left\{
                tn_2
                +
                t^2N_2
                -
                Mpt(2-p)
                +
                \zeta\log\frac{eMt}{\zeta}
                +
                o(\zeta)
            \right\}
            \ind{n_1>I_pn}
        \right]
        &=
        o(1).
        \label{eq:asym_dense_wxy_tail_second}
\end{align}
\end{lemma}

\begin{proof}
This is the final permutation-summation step in
\cite[Appendix A.3]{wu2023testing}, with $s^2$ replaced everywhere by
$t=s_1s_2$ and with $I_p$ in place of
$p(\log(1/p)-1+p)$.  We briefly indicate why the cited argument applies
without further changes.

The first bound is obtained from \cite[Proposition~2, Eq.~(48)]{wu2023testing}
with
\[
        \mu=t/2,\qquad
        \nu=0,\qquad
        \tau=t,\qquad
        a=0,\qquad
        b=I_p n.
\]
Indeed, $b=n^{1-o(1)}=\omega(1)$ and, by
\eqref{eq:asym_dense_sharp_lower_bound},
\[
        \mu b+\nu+2-\log b
        =
        \frac12ntI_p+2-\log(I_pn)
        =
        -\frac{\epsilon+o(1)}2\log n+2
        \le0.
\]

For the second bound,~\cite{wu2023testing} split the range $I_pn\le n_1\le n$ into
\[
        \beta n\le n_1\le n,\qquad
        \beta'n\le n_1\le\beta n,\qquad
        I_pn\le n_1\le\beta'n,
\]
where
\[
        \beta=\frac{\log^2(npt)}{npt},
        \qquad
        \beta'=\frac{\log(npt)}{100\,npt}.
\]
By \Cref{lem:asym_dense_W_function}, which is
\cite[Lemma~7]{wu2023testing} with $s^2$ replaced by $t$, the estimates
for $\zeta$ in these three ranges are identical to those in
\cite[Appendix A.3]{wu2023testing}.  The subsequent applications of
\cite[Proposition~2, Eq.~(47)]{wu2023testing} are also unchanged after the
same substitution.  In the last range, the only additional input is the deterministic estimate
\[
        \max_{I_pn\le k\le\beta'n}\psi(k)=o(\log(I_pn)),
\]
proved in \cite[Appendix A.3, Case~2(c)]{wu2023testing}; that calculation
uses only the boundary relation
\eqref{eq:asym_dense_sharp_lower_bound}, the assumption $p=n^{-o(1)}$,
and the product $t=s_1s_2$, not $s_1$ and $s_2$ separately.  Therefore the
second expectation is $o(1)$, as claimed.
\end{proof}

Combining \eqref{eq:asym_dense_two_terms} with
\Cref{lem:asym_dense_wxy_tail}, we obtain
\(
        \Expect_\Q[(L')^2]\le 1+o(1).
\)
Together with \Cref{lem:asym_dense_cond_high_prob} and \Cref{lem:conditional_second_moment}, this completes the proof of \Cref{thm:asym_er_weak_detection_converses-iii}.

\section{Converse proofs for alignment in the asymmetric setting}
\label{app:asym_alignment_proofs}
\label{app:converses_asym_recovery}

\subsection{Proof of~\texorpdfstring{\Cref{thm:asym_dd_recovery-ii}}{} }
\label{app:asym_dd_recovery}

We shall explain how the proof of the alignment converse in~\cite[Theorem 1.1, Eq. (1.3)]{ding2023densesubgraph} can be extended to prove the asymmetric version,~\texorpdfstring{\Cref{thm:asym_dd_recovery-ii}}{}.  We do not attempt to summarize the proof; we only describe the necessary minor modifications.
We use the notation 
\(
        t\coloneqq s_1s_2,
\)
and
\(
        \lambda_n\coloneqq npt,
\)
and the single-edge likelihood ratio $\ell$ given by
\eqref{eq:asym_single_edge_lr}.  Throughout, set $\lambda_\alpha\coloneqq \varrho^{-1}(1/\alpha)$ and assume
\[
        p=n^{-\alpha+o(1)},
        ~~~~~~ 
        \alpha\in(0,1),
        ~~~~~~
       \max\{s_1, s_2\} \to 0 ,
        ~~~~~~
        \lambda_n\leq \lambda_\alpha-\epsilon.
\]
The proof of~\cite{ding2023densesubgraph} goes through with minor changes by replacing $s^2$ by $t=s_1s_2.$  In particular all the lemmas and propositions in~\cite[Section 3 and the Appendix]{ding2023densesubgraph} hold in the asymmetric case with only minor changes in some proofs, as we describe.

The proof of~\cite[Proposition 3.2]{ding2023densesubgraph} needs to be slightly modified as follows.
In the asymmetric case, the posterior probability distribution of $\pistar$ given in ~\cite[Eq. (3.14)]{ding2023densesubgraph} becomes (in their notation)
\begin{align*}
    \frac{{\cal Q}[\pi, G , \G]}{\P[G,\G]} =
    \frac 1 {n!} \prod_{e\in E_0} \ell(G_e,\G_{\Pi(e)}) = 
    \frac {C_{11}^{|\calE_{\pi}|} C_{10}^{|E|} C_{01}^{|\E|}C_{00}^{\binom n 2}}{n!}
\end{align*}
where the $C_{ij}$ are defined in our Appendix~\ref{app:asym_dense_detection}.
This expression propagates in a straightforward way to the definition of $f(G,\G,A,\sigma)$ in~\cite{ding2023densesubgraph}.
Another minor change in the proof of~\cite[Proposition 3.2]{ding2023densesubgraph} is that the justification for $\TV(\P,\Q)\to 0$ should be changed to our \Cref{thm:asym_er_weak_detection_converses-ii}, suitable for unequal subsampling probabilities. 

The only other proposition in~\cite[Section 3]{ding2023densesubgraph} that requires a proof modification is~\cite[Proposition 3.6]{ding2023densesubgraph}. In the appendix of~\cite{ding2023densesubgraph}, the proposition is split into two parts, namely Propositions A.7 and A.8.    The proof of Proposition A.7 is given in the supplementary materials
of~\cite{ding2023densesubgraph} (incorrectly labeled ``Proof of Proposition A.1'').  This is where the constants $C_{ij}$ come in.   The proof goes through as before with $s^2$ replaced by $t=s_1s_2$ because 
\[
    C_{11} = (1 + o(1))\frac 1 p,
    ~~~~
    C_{10}\leq 1,
    ~~~~
    C_{01}\leq 1,
    ~~~~ 
    \mbox{and }
    C_{00}^{\binom K 2} = (1+O(pt))^{\binom K 2} = e^{o(K)}
\]
(whereas the original proof uses the facts $P = (1 + o(1))\frac 1 p$,
$Q\leq 1$, $R^{\binom K 2} = e^{o(K)}$ and $K = n^\beta$ for the variables $P,Q,R,K$ defined in 
\cite{ding2023densesubgraph}.)

It remains to specify changes to proofs in the appendix of~\cite{ding2023densesubgraph} for the proof of Proposition A.8.
Only one proposition in the appendix needs a proof modification -- that is ~\cite[Proposition A.9]{ding2023densesubgraph}.  The proof goes through in the asymmetric case with the help of \Cref{lem:asym_dd_cond_orbit} 
of our paper. In particular, our~\eqref{eq:asym_conditional_orbit} yields~\cite[Eq. (A4)]{ding2023densesubgraph} and the fact 
$\prod_{e\in O}\ell(X^1_e,X^2_{{\pistar}\circ\sigma(e)})=p^{-r}$
for fully occupied orbits $O$ gives~\cite[Eq. (A5)]{ding2023densesubgraph}.

\subsection{Proof of~\texorpdfstring{\Cref{thm:asym_dd_recovery-i,thm:asym_dd_recovery-iii}}{}} 
\label{app:asym_sparse_dense_partial_recovery}

Throughout this subsection, set
\(
        t\coloneqq s_1s_2,
\)
and
\(
        N\coloneqq |\En|=\binom n2 .
\)
By symmetry between the two graphs, we may assume without loss of generality that
\(
        s_1\ge s_2 .
\)
Let
\[
        A\coloneqq X^1,
        \qquad
        B\coloneqq X^2,
        \qquad
        q\coloneqq ps_1,
        \qquad
        r\coloneqq ps_2,
        \qquad
        s\coloneqq s_2 .
\]
Thus \(r\le s\).  For a permutation \(\pi\in\Sn\) and vertex pair $e \in \En$, define
\(
        (A^\pi)_e \coloneqq A_{\pi^{-1}(e)},
\)
so that, under the planted law, the pairs
\((A^{\pistar}_e,B_e)_{e\in\En}\) are independent.

For \(\theta\in[r,s]\), define an interpolating edge law \(P_\theta\) by
\[
        P_\theta(1,1)=q\theta,
        \qquad
        P_\theta(1,0)=q(1-\theta),
        \qquad
        P_\theta(0,1)=r-q\theta,
        \qquad
        P_\theta(0,0)=1-q-r+q\theta .
\]
Equivalently, under $\P_{\theta},$ $A$ and $B$ have marginals \(\Bern(q)\) and \(\Bern(r)\) and 
\[
        \P_\theta(B=1\mid A=1)=\theta,
        \qquad
        \P_\theta(B=1\mid A=0)
        =
        \eta_\theta
        \coloneqq
        \frac{r-q\theta}{1-q}.
\]
At \(\theta=r\), \(A\) and \(B\) are independent.  
At \(\theta=s\), we get the planted
asymmetric \erdosrenyi\ law with parameters \(p,s_1,s_2\).

Let \(I_s(A,B;\pistar)\) denote the mutual information between $\pi$ and the observation $(A,B)$, at the endpoint $\theta=s$. Further, let $I_{\rm e}$ denote the one-edge mutual information between one correctly aligned pair under the law $P_s$. Finally let
\[
        \operatorname{mmse}_\theta(A^{\pistar})
        \coloneqq
        \Expect_\theta
        \left[
            \left\|
                A^{\pistar}
                -
                \Expect_\theta[A^{\pistar}\mid A,B]
            \right\|^2
        \right]
\]
denote the minimum mean square error (MMSE) of estimating $A^{\pistar}$ based on $(A,B)$ distributed according to $P_{\theta}$.

\medskip

\begin{lemma}[Asymmetric area theorem]
\label{lem:asym_recovery_area}
\[
\begin{aligned}
        I_s(A,B;\pistar)
        \le\,
        N I_{\rm e}
        + N q s^2
        +
        \int_r^s
        \frac{\theta-r}{s(1-q)^2}
        \left(
            \operatorname{mmse}_\theta(A^{\pistar})
            -
            Nq(1-q)
        \right)
        \dd\theta .
\end{aligned}
\]
\end{lemma}

\begin{proof}
We follow \cite[Proof of Proposition~3]{wu2022settling}, recording the changes caused by asymmetric subsampling.  Let $X=A^{\pistar}$, and define
\[
        g(\theta)
        \coloneqq
        D\!\left(P_\theta\middle\Vert
        \Bern(q)\otimes\Bern(r)\right)
        =q\,d(\theta\Vert r)+(1-q)d(\eta_\theta\Vert r),
\]
where \(d(\cdot\Vert\cdot)\) is binary relative entropy.  Thus
\(g(r)=0\) and \(g(s)=I_{\rm e}\).  Since \(A - A^{\pistar} - B\) forms a Markov chain, we have
\(
        I_\theta(A,B;\pistar)
        =I_\theta(A^{\pistar};B\mid A).
\)

For \(e\in\En\), let $B_{-e}$ denote the vector $B$ excluding $B_e$ and define
\[
        x_e^\circ
        \coloneqq
        \Expect_\theta[X_e\mid B_{-e},A],
        \qquad
        \widehat x_e
        \coloneqq
        \Expect_\theta[X_e\mid B,A],
\]
and set
\[
        y_e
        \coloneqq
        P_\theta(B_e=1\mid B_{-e},A)
        =
        r+\frac{\theta-r}{1-q}(x_e^\circ-q).
\]
Writing \(h(x)=-x\log x-(1-x)\log(1-x)\), differentiation of the conditional entropy $H_{\theta}(A^{\pistar} | A,B)$ using methods similar to \cite[Lemma 7.1]{deshpande2017asymptotic} gives
\[
        I_s(A,B;\pistar)
        =
        N I_{\rm e}
        -\int_r^s \mathcal I_\theta\,\dd\theta,
        ~~~~ \mbox{where} ~~~~
        \mathcal I_\theta
        =-
        \sum_{e\in\En}
        \Expect_\theta\left[
            \frac{\partial y_e}{\partial\theta}h'(y_e)
        \right].
\]
This is the entropy-derivative identity in~\cite[Proof of Proposition~3]{wu2022settling}\footnote{This derivation applies verbatim to the asymmetric case, since it does not use symmetry between $A$ and $B$. Therefore, the reader is directed to~\cite[eq. (24)--(26)]{wu2022settling} for the full derivation.}; here
\[
        \frac{\partial y_e}{\partial\theta}
        =\frac{x_e^\circ-q}{1-q},
        \qquad
        y_e-r=\frac{\theta-r}{1-q}(x_e^\circ-q).
\]
By the mean-value theorem and the fact $\Expect_{\theta}\big[\frac{\partial y_e}{\partial \theta}\big]=0$, for some \(\xi_e\) between \(y_e\) and \(r\),
\begin{align*}
        -\Expect_\theta\left[
            \frac{\partial y_e}{\partial\theta}h'(y_e)
        \right]
        & =
        \frac{\theta-r}{(1-q)^2}
        \Expect_\theta\left[
            \frac{(x_e^\circ-q)^2}{\xi_e(1-\xi_e)}
        \right] 
        \ge
        \frac{\theta-r}{s(1-q)^2}
        \Var_\theta(x_e^\circ).
\end{align*}
Here \(\Expect_\theta[x_e^\circ]=q\), and
\(\eta_\theta\le y_e,r\le\theta\le s\), so \(\xi_e\le s\).  Consequently,
\[
        \mathcal I_\theta
        \ge
        \sum_{e\in\En}
        \frac{\theta-r}{s(1-q)^2}
        \Var_\theta(x_e^\circ).
\]
It remains to relate the variance of $x_e^{\circ}$ to the MMSE.  Bayes' rule gives
\[
        \widehat x_e
        =
        \begin{cases}
        \displaystyle
        \frac{(1-\theta)x_e^\circ}
        {1-\eta_\theta-x_e^\circ(\theta-\eta_\theta)},
            & B_e=0,\\[1.2em]
        \displaystyle
        \frac{\theta x_e^\circ}
        {\eta_\theta+x_e^\circ(\theta-\eta_\theta)},
            & B_e=1.
        \end{cases}
\]
Since \(\eta_\theta\le\theta\le s\),
\[
        \widehat x_e
        \le
        (1-B_e)x_e^\circ
        +
        B_e\min\left\{1,\frac{s}{\eta_\theta}x_e^\circ\right\}.
\]
Conditionally on \(X_e\), the variables \(B_e\) and \(x_e^\circ\) are
independent.  Hence, using \(\theta\le s\) and
\(\min\{1,z\}^2\le z\),
\[
\begin{aligned}
&\Expect_\theta\left[
 B_e\min\left\{1,\frac{s}{\eta_\theta}x_e^\circ\right\}^2
 \right]
 \le s \, \Expect_\theta[X_e]+s \, \Expect_\theta[x_e^\circ]
 =2sq,\\
&\Expect_\theta[(x_e^\circ)^2]
 =q^2+\Var_\theta(x_e^\circ).
\end{aligned}
\]
Since $B_e(1-B_e)\equiv 0$, it follows that
\(
        \Expect_\theta[\widehat x_e^2]
        \le q^2+\Var_\theta(x_e^\circ)+2sq.
\)
Therefore
\[
\begin{aligned}
        \operatorname{mmse}_\theta(A^{\pistar})
        &=
        \sum_{e\in\En}
        \Expect_\theta
        \left[
            (A^{\pistar}_e-\widehat x_e)^2
        \right] 
        =
        \sum_{e\in\En}
        \left(q-\Expect_\theta[\widehat x_e^2]\right)            \ge
        Nq(1-q)
        -
        \sum_{e\in\En}\Var_\theta(x_e^\circ)
        -
        2Nsq .
\end{aligned}
\]
Combining the last display with the lower bound on \(\mathcal I_\theta\)
and integrating yields the claimed integral term.  The remaining term is
\[
        2Nsq\int_r^s
        \frac{\theta-r}{s(1-q)^2}\,\dd\theta
        =Nq\frac{(s-r)^2}{(1-q)^2}
        \le Nqs^2,
\]
where the last inequality follows from \(r\ge qs\).
\end{proof}

\medskip 

\begin{lemma}[MMSE-to-overlap conversion]
\label{lem:asym_recovery_mmse_overlap}
There is a universal constant \(C<\infty\) such that the following holds.  If,
for some \(0\le\xi\le1\),
\[
        \operatorname{mmse}_\theta(A^{\pistar})
        \ge
        Nq(1-q)(1-\xi),
\]
then every estimator \(\widehat\pi=\widehat\pi(A,B)\) satisfies
\[
        \Expect_\theta[\ov(\pistar,\widehat\pi)]
        \le
        C\left\{
            (q+\xi)^{1/4}
            +
            \left(
                \frac{n\log n}{Nq}
            \right)^{1/4}
        \right\}.
\]
\end{lemma}

\begin{proof}
Since \(\Expect\|A\|^2=Nq\) and
\[
        Nq(1-q)(1-\xi)
        =Nq(1-\widetilde\xi),
        \qquad
        \widetilde\xi=q+\xi-q\xi\le q+\xi,
\]
the result follows directly from
\cite[Proposition~4]{wu2022settling}, applied to the uncentered adjacency
vector \(A\).  That proposition uses only the law of \(A\) and the relation
between fixed vertices and fixed edges; it does not use equality of the two
graph marginals.
\end{proof}

\medskip

\begin{lemma}
\label{lem:asym_conditional_chi2_to_kl}
Let \(P_{XY}\) be the joint law of \((X,Y)\), and suppose $\calE$ is an event independent of $X$ such that $P(\calE) = 1 - \delta.$
If \(P_{Y\mid X}\ll Q_Y\) almost surely, then
\begin{align}
        D(P_Y\Vert Q_Y)
        \le{}&
        \log\!\left(1+\chi^2(P_{Y\mid\mathcal E}\Vert Q_Y)\right)
        \nonumber\\
        &+\delta\left(
            \log\frac1\delta
            +\Expect \big[ 
                D(P_{Y\mid X}\Vert Q_Y)
            \big]
        \right)
        +\sqrt{\delta\,
        \Var\!\left(
            \log\frac{\dd P_{Y\mid X}}{\dd Q_Y}(Y)
        \right)}.
        \label{eq:asym_conditional_chi2_to_kl}
\end{align}
\end{lemma}

\begin{proof}
This is \cite[Lemma~1]{wu2022settling} and its proof is unchanged for the asymmetric case.
\end{proof}

For later use, let \(\ell\) be the one-edge likelihood ratio at an endpoint with parameters \(p,s_1,s_2\), as defined in~\eqref{eq:asym_single_edge_lr}. 
When \(p\le1-\Omega(1)\) and
\(\max\{s_1,s_2\}=o(1)\), direct expansion gives
\begin{align}
        I_{\rm e}
        &=pt\left(\log\frac1p-1+p\right)
          +O\!\left(pt(s_1+s_2)\right),
        \label{eq:asym_edge_mi_expansion}\\
        \Var_{P_s} \big(\log\ell \big)
        &\le Cpt\left(1+\log^2\frac1p\right).
        \label{eq:asym_edge_llr_variance}
\end{align}
These estimates follow by using the definition of $\ell(a,b)$ for $a,b\in\{0,1\}$, and using
\(\log(1-x)=-x+O(x^2)\).

We will repeatedly use the following consequence of the area theorem.
For the endpoint model with parameters \(p,s_1,s_2\), we may write
\begin{align}
        I_s(A,B;\pistar)
        =
        N I_{\rm e}-\zeta_n ,
        \label{eq:I_s-formula}
\end{align}
where $\zeta_n$ will be appropriately bounded later. Let \(\theta_0=(1-\delta)s\), where \(\delta\in(0,1)\) is fixed and
\(\theta_0>r\).  Note that the MMSE is nonincreasing\footnote{Indeed, if \(r\le\theta_1<\theta_2\le s\), set
\(
        \kappa=(\theta_1-r)/(\theta_2-r).
\)
Starting from the \(\theta_2\)-experiment, pass every \(B_e\) independently
through the binary channel
\[
        \P(B'_e=1\mid B_e=1)=r+\kappa(1-r),
        \qquad
        \P(B'_e=1\mid B_e=0)=r(1-\kappa).
\]
A direct calculation gives
\(
        \P(B'_e=1\mid A^{\pistar}_e=1)=\theta_1
\)
and
\(
        \P(B'_e=1\mid A^{\pistar}_e=0)=\eta_{\theta_1}.
\)
Thus \((A,B')\) has the \(\theta_1\)-law and is a post-processing of the
\(\theta_2\)-experiment, proving the monotonicity.
} in \(\theta\). It follows that the quantity
\[
        \Delta_\theta
        \coloneqq
        Nq(1-q)-\operatorname{mmse}_\theta(A^{\pistar})
\]
is nondecreasing in \(\theta\).  Combining~\eqref{eq:I_s-formula} with
\Cref{lem:asym_recovery_area} gives
\[
        \int_r^s
        \frac{\theta-r}{s(1-q)^2}
        \Delta_\theta\,\dd\theta
        \le
        \zeta_n+Nqs^2 .
\]
Therefore
\[
        \Delta_{\theta_0}
        \int_{\theta_0}^{s}
        \frac{\theta-r}{s(1-q)^2}\,\dd\theta
        \le
        \zeta_n+Nqs^2 .
\]
If \(p\le1-\Omega(1)\), then for fixed sufficiently small \(\delta>0\), we have
\(
        \int_{\theta_0}^{s}
        \frac{\theta-r}{s(1-q)^2}\,\dd\theta
        =
        \Theta(s).
\)
Hence
\begin{align}
\label{eq:asym_common_mmse_lower}
        \operatorname{mmse}_{\theta_0}(A^{\pistar})
        \ge
        Nq(1-q)
        \Big(
            1
            -
            O(s)
            -
            O\Big(\frac{\zeta_n}{Nqs}\Big)
        \Big).
\end{align}

\subsubsection{Sparse regime: proof of~\texorpdfstring{\Cref{thm:asym_dd_recovery-i}}{}}

We will use the following sparse mutual-information input.

\medskip 

\begin{lemma}[Sparse mutual-information input]
\label{lem:asym_sparse_recovery_mi_input}
Assume
\(
        p=n^{-\Omega(1)},
        \max\{s_1,s_2\}\to0,
        np=\omega(\log^2 n),
\)
and
\(
        nps_1s_2\le 1-\epsilon
\)
for a fixed \(\epsilon>0\).  Then
\[
        I_s(A,B;\pistar)
        =
        N I_{\rm e}-\zeta_n,
        \qquad
        \zeta_n=O(\log n).
\]
\end{lemma}

\begin{proof}
Let \(\Q_{A,B}=\P_A\otimes\P_B\).  Since, conditionally on \(\pistar\),
the \(N\) aligned edge pairs are independent and have one-edge mutual
information \(I_{\rm e}\), the chain rule gives
\[
        I_s(A,B;\pistar)
        =
        N I_{\rm e}-D(\P_{A,B}\Vert\Q_{A,B}).
\]
Thus it remains to show that the last divergence is \(O(\log n)\).

Let \(\mathcal A_{\rm pf}\) be the pseudoforest event from
Appendix~\ref{app:asym_sparse_detection}, and set
\(
        \delta_n=\P(\mathcal A_{\rm pf}^c).
\)
Since \(c_n\leq1-\epsilon\), \Cref{lem:asym_er_pseudoforest}
gives
\(
    \delta_n
    \leq \frac{2}{n\epsilon^3}
    =O(n^{-1}).
\)
Moreover,
\Cref{prop:asym_conditional_lr} gives
\[
        \chi^2(\P_{A,B\mid\mathcal A_{\rm pf}}\Vert\Q_{A,B})
        \le
        \frac{G_{\rm pf}(s_1s_2)}
        {\P(\mathcal A_{\rm pf})^2}-1
        =
        O(s_1s_2+n^{-1})
        =
        o(1).
\]
By permutation equivariance,
\(
        \P(\mathcal A_{\rm pf}\mid\pistar=\pi)
        =\P(\mathcal A_{\rm pf})
\)
for every \(\pi\), so this event is independent of \(\pistar\).
For a fixed alignment, let
\[
        Z
        \coloneqq
        \log\frac{\dd\P_{A,B\mid\pistar}}{\dd\Q_{A,B}}(A,B).
\]
This is a sum of \(N\) independent one-edge log-likelihood ratios.
By \eqref{eq:asym_edge_mi_expansion} and
\eqref{eq:asym_edge_llr_variance},
\(nps_1s_2=O(1)\), and \(\log(1/p)=O(\log n)\), we have
\[
        \Expect Z=N I_{\rm e}=O(n\log n),
        \qquad
        \Var(Z)=O(n\log^2 n).
\]
Applying \Cref{lem:asym_conditional_chi2_to_kl} with
\(X=\pistar\), \(Y=(A,B)\), and
\(\mathcal E=\mathcal A_{\rm pf}\), we obtain
\[
\begin{aligned}
        D(\P_{A,B}\Vert\Q_{A,B})
        &\le
        o(1)
        +O(n^{-1})\,O(n\log n)
        +\sqrt{O(n^{-1})\,O(n\log^2 n)}
        =O(\log n).
\end{aligned}
\]
This proves the lemma.
\end{proof}

By monotonicity under independent thinning, we may first increase the
retention probabilities to a fixed boundary point.  Reducing \(\epsilon\)
if necessary, assume \(0<\epsilon<1\), and set
\[
        T_n\coloneqq\frac{1-\epsilon/2}{np},
        \qquad
        \bar s_1\coloneqq\max\{s_1,\sqrt{T_n}\},
        \qquad
        \bar s_2\coloneqq\frac{T_n}{\bar s_1}.
\]
Since \(s_1\ge s_2\) and
\(
        s_1s_2\le(1-\epsilon)/(np)<T_n,
\)
we have \(\bar s_i\ge s_i\) for \(i=1,2\),
\(\bar s_1\ge\bar s_2\), and \(\bar s_1\bar s_2=T_n\).  These
inequalities follow directly by considering separately
\(s_1\le\sqrt{T_n}\) and \(s_1>\sqrt{T_n}\).
Moreover, \(T_n\to0\) and \(\max_i s_i\to0\), so
\(\max_i\bar s_i\to0\).  The original experiment is obtained from this
more informative experiment by retaining each observed edge in graph \(i\)
independently with probability \(s_i/\bar s_i\).  It is therefore enough
to prove intractability for the latter experiment.

Relabeling \((\bar s_1,\bar s_2)\) as \((s_1,s_2)\), we henceforth assume
\(
        nps_1s_2=1-{\epsilon}/{2}.
\)
Choose a fixed \(\delta>0\) sufficiently small that
\(
        \frac{1-\epsilon/2}{1-\delta}
        \le 1- {\epsilon}/{4},
\)
and define the inflated endpoint parameters
\[
        p'\coloneqq (1-\delta)p,
        \qquad
        s_i'\coloneqq \frac{s_i}{1-\delta},
        \qquad i=1,2.
\]
Since \(\max_i s_i\to0\), these are valid probabilities for all
sufficiently large \(n\).  The marginals are unchanged, i.e.
\(
        p's_i'=ps_i
\)
for
\(
        i=1,2.
\)
Moreover,
\[
        np's_1's_2'
        =
        \frac{nps_1s_2}{1-\delta}
        \le
        1-\frac{\epsilon}{4}.
\]
Also \(p'=n^{-\Omega(1)}\), \(np'=\omega(\log^2n)\), and
\(\max_i s_i'\to0\).
We apply the interpolation above to this inflated endpoint.  Thus
\[
        q=p's_1'=ps_1,
        \qquad
        r=p's_2'=ps_2,
        \qquad
        s=s_2'=\frac{s_2}{1-\delta}.
\]
The point
\(
        \theta_0=(1-\delta)s=s_2
\)
is exactly the saturated model before inflation, because
\(
        q\theta_0=ps_1s_2,
\)
and
\(
        r-q\theta_0=ps_2(1-s_1).
\)
By \Cref{lem:asym_sparse_recovery_mi_input}, the inflated endpoint satisfies
\(
        I_s(A,B;\pistar)
        =
        N I_{\rm e}-O(\log n).
\)
Since \(s=s_2'\le \sqrt{s_1's_2'}=O((np)^{-1/2})=o(1)\), and
\(
        Nqs
        =Np's_1's_2'
        =\Theta(n),
\)
\eqref{eq:asym_common_mmse_lower} gives
\[
        \operatorname{mmse}_{\theta_0}(A^{\pistar})
        \ge
        Nq(1-q)(1-o(1)).
\]
Here \(q=p's_1'=ps_1=o(1)\).  Thus the preceding display has the form
required by \Cref{lem:asym_recovery_mmse_overlap}, with \(\xi=o(1)\).
Furthermore, since \(s_1'\ge s_2'\),
\[
        nq
        =
        np's_1'
        \ge
        np'\sqrt{s_1's_2'}
        =
        \sqrt{np'\cdot np's_1's_2'}
        =
        \omega(\log n).
\]
Hence
\(
        \frac{n\log n}{Nq}
        =
        O\left(\frac{\log n}{nq}\right)
        =
        o(1).
\)
Applying \Cref{lem:asym_recovery_mmse_overlap} at
\(\theta=\theta_0\), we obtain
\(
        \Expect_{\theta_0}[\ov(\pistar,\widehat\pi)]
        =
        o(1)
\)
for every estimator \(\widehat\pi\).  Markov's inequality then gives, for every
fixed \(a>0\),
\[
        \P_{\theta_0}\big(\ov(\pistar,\widehat\pi)\ge a\big)=o(1).
\]
This proves
\Cref{thm:asym_dd_recovery-i}.

\subsubsection{Dense regime: proof of~\texorpdfstring{\Cref{thm:asym_dd_recovery-iii}}{}}

Let
\(
        f_p\coloneqq \log\frac1p-1+p .
\)
The dense hypothesis may be written, after decreasing \(\epsilon\) if needed, as
\(
        nps_1s_2 f_p
        \le
        (2-\epsilon)\log n .
\)

\medskip 

\begin{lemma}[Dense mutual-information input]
\label{lem:asym_dense_recovery_mi_input}
Assume
\(
        p\le1-\Omega(1),
        p=n^{-o(1)},
\) and
\(
        \max\{s_1,s_2\}\to0.
\)
Further, for \(i=1,2\), assume that
\(
        np^2s_i=\omega(1)
\)
and
\[
        nps_1s_2\left(\log\frac1p-1+p\right)
        \le
        (2-\epsilon)\log n .
\]
Then
\(
        I_s(A,B;\pistar)
        =
        N I_{\rm e}-o(1),
\) 
and
\begin{align} \label{eq:dense-Ie-expansion}
        I_{\rm e}
        =
        ps_1s_2
        \left(\log\frac1p-1+p\right)(1+o(1)).
\end{align}
\end{lemma}

\begin{proof}
Let \(\Q_{A,B}=\P_A\otimes\P_B\).  As above, the chain rule gives
\(
        I_s(A,B;\pistar)=N I_{\rm e}-D(\P_{A,B}\Vert\Q_{A,B})
\)
and we set
\(
        \zeta_n=D(\P_{A,B}\Vert\Q_{A,B}).
\)

First reduce to the boundary case.  If the displayed threshold in the
lemma is strict, put
\[
        T_n=\frac{(2-\epsilon)\log n}
        {np(\log(1/p)-1+p)},\qquad
        \bar s_1=\max\{s_1,\sqrt{T_n}\},\qquad
        \bar s_2=\frac{T_n}{\bar s_1}.
\]
Then \(\bar s_i\ge s_i\), \(\bar s_1\bar s_2=T_n\), and
\(\max_i\bar s_i\to0\), since \(T_n=n^{-1+o(1)}\).  Also
\(np^2\bar s_i=\omega(1)\).  The data processing inequality gives
\[
        D(\P_{A,B}\Vert\Q_{A,B})
        \le D(\overline\P_{A,B}\Vert\overline\Q_{A,B}).
\]
It therefore suffices to consider the boundary model, which we relabel by
\((s_1,s_2)\).

Let \(\mathcal E(\pistar)\) be the conditioning event from
Appendix~\ref{app:asym_dense_detection}, and write
\(
        I_p=p(\log(1/p)-1+p).
\)
By \Cref{lem:asym_dense_cond_high_prob} and the conditional second-moment
calculation in~Appendix~\ref{app:asym_dense_detection},
\[
        \delta_n
        \coloneqq\P(\mathcal E(\pistar)^c)
        \le e^{-\Omega(I_pn)}
        =e^{-n^{1-o(1)}},
        \qquad
        \chi^2(\P_{A,B\mid\mathcal E(\pistar)}
        \Vert\Q_{A,B})=o(1).
\]
This calculation uses the individual \(s_i\)'s only in the two edge-count
concentration estimates in
\eqref{eq:asym_dense_edge_count_concentration}; the present assumptions
\(np^2s_i=\omega(1)\) are exactly what those estimates require.  All other
steps depend on \(s_1,s_2\) only through \(t=s_1s_2\).  By permutation
equivariance, \(\mathcal E(\pistar)\) is independent of \(\pistar\).

For a fixed alignment, set
\(
        Z
        \coloneqq
        \log\frac{\dd\P_{A,B\mid\pistar}}{\dd\Q_{A,B}}(A,B).
\)
By \eqref{eq:asym_edge_mi_expansion},
\eqref{eq:asym_edge_llr_variance}, and the boundary relation,
\[
        \Expect Z=N I_{\rm e}=O(n\log n),
        \qquad
        \Var(Z)
        =
        O\!\left(
            Nps_1s_2\left(1+\log^2\frac1p\right)
        \right)
        =
        n^{1+o(1)}.
\]
Applying \Cref{lem:asym_conditional_chi2_to_kl} with
\(X=\pistar\), \(Y=(A,B)\), and
\(\mathcal E=\mathcal E(\pistar)\) gives
\(
        \zeta_n=o(1).
\)
Finally, independently of the boundary reduction, applying
\eqref{eq:asym_edge_mi_expansion} to the original parameters,
using \(\max_i s_i=o(1)\) and the fact that
\(\log(1/p)-1+p\) is bounded away from zero, yields~\eqref{eq:dense-Ie-expansion}.
\end{proof}

We now prove the dense converse.  We first reduce to a boundary model.  Set
\[
        T_n\coloneqq
        \frac{(2-\epsilon/2)\log n}{npf_p},
        \qquad
        \bar s_1\coloneqq\max\{s_1,\sqrt{T_n}\},
        \qquad
        \bar s_2\coloneqq\frac{T_n}{\bar s_1}.
\]
As in the sparse case, \(\bar s_i\ge s_i\),
\(\bar s_1\ge\bar s_2\), and \(\bar s_1\bar s_2=T_n\).  Since
\(p=n^{-o(1)}\) and \(p\le1-\Omega(1)\), we have
\(T_n=n^{-1+o(1)}\), and hence \(\max_i\bar s_i\to0\).  Also
\(np^2\bar s_i=\omega(1)\), because \(\bar s_i\ge s_i\).  Thus it is
enough, by independent thinning, to prove the result for this more
informative model.  Relabeling $(\bar s_1, \bar s_2)$ as $(s_1,s_2)$, we have
\(
        nps_1s_2f_p=(2-\epsilon/2)\log n.
\)
Choose a fixed small \(\delta>0\) and define
\[
        p'\coloneqq (1-\delta)p,
        \qquad
        s_i'\coloneqq \frac{s_i}{1-\delta},
        \qquad i=1,2.
\]
Since \(\max\{s_1,s_2\}\to0\), \(s_i'\le1\) for all sufficiently large \(n\).
Also
\(
        p's_i'=ps_i,
\)
and
\(                  
    p's_1's_2'=\frac{ps_1s_2}{1-\delta}.
\)
Fix \(c>0\) such that \(p\le1-c\) for all sufficiently large \(n\).
The choice of \(\delta\) can be made so that
\[
        \frac{f_{(1-\delta)p}}{(1-\delta)f_p}
        \le
        \frac{2-\epsilon/4}{2-\epsilon/2}.
\]
Indeed,
\(
        f_{(1-\delta)p}
        =f_p-\log(1-\delta)-\delta p
\)
and \(f_p\) is bounded away from zero uniformly for \(p\le1-c\), so the
ratio on the left is \(1+O(\delta)\).  Consequently,
\[
        np's_1's_2'f_{p'}
        =
        nps_1s_2f_p\,
        \frac{f_{p'}}{(1-\delta)f_p}
        \le
        (2-\epsilon/4)\log n.
\]
Moreover,
\(
        np'^2s_i'=(1-\delta)np^2s_i=\omega(1).
\)
We apply the interpolation to this inflated endpoint, so
\[
        q=p's_1'=ps_1,
        \qquad
        r=p's_2'=ps_2,
        \qquad
        s=s_2'=\frac{s_2}{1-\delta}.
\]
The point
\(
        \theta_0=(1-\delta)s=s_2
\)
corresponds exactly to the saturated model before inflation.

By \Cref{lem:asym_dense_recovery_mi_input},
\(
        I_s(A,B;\pistar)=N I_{\rm e}-o(1).
\)
Using \eqref{eq:asym_common_mmse_lower},
\[
        \operatorname{mmse}_{\theta_0}(A^{\pistar})
        \ge
        Nq(1-q)
        \left(
            1
            -
            O(s)
            -
            o\left(\frac1{Nqs}\right)
        \right).
\]
Now
\(
        Nqs
        =Np's_1's_2'
        =
        \Theta\left({n\log n}/{f_p}\right)
        \to\infty,
\)
where the saturation relation was used.  Also \(s=s_2'=o(1)\) and
\(q=p's_1'=ps_1=o(1)\).  Hence
\(
        \operatorname{mmse}_{\theta_0}(A^{\pistar})
        \ge
        Nq(1-q)(1-o(1)).
\)
This is the hypothesis of \Cref{lem:asym_recovery_mmse_overlap} with
\(\xi=o(1)\).
Moreover, because \(s_1'\ge s_2'\),
\[
        nq
        =
        np's_1'
        \ge
        np'\sqrt{s_1's_2'}
        =
        \sqrt{np'\cdot np's_1's_2'}
        =
        n^{1/2-o(1)}.
\]
Thus
\(
        \frac{n\log n}{Nq}
        =
        O\left(\frac{\log n}{nq}\right)
        =
        o(1).
\)
Applying \Cref{lem:asym_recovery_mmse_overlap} at \(\theta=\theta_0\), we get
\(
        \Expect_{\theta_0}[\ov(\pistar,\widehat\pi)]
        =
        o(1)
\)
for every estimator \(\widehat\pi\).  Markov's inequality yields, for every
fixed \(a>0\),
\[
        \P_{\theta_0}\big(\ov(\pistar,\widehat\pi)\ge a\big)=o(1).
\]
This proves
\Cref{thm:asym_dd_recovery-iii}.

\section{Proof of~\texorpdfstring{\Cref{thm:ER-GLRT-achievability}}{}} \label{app:ER-GLRT-achievability}

For an alignment $\pi=(\pi_1,\cdots,\pi_m)\in\Sn^{m}$ with $\pi_1=\Id$, define the pairwise-overlap score
\[
        T_\pi
        \coloneqq
        \sum_{e\in\En}
        \sum_{1\le k<\ell\le m}
        X^{k}_{\pi_k(e)}X^{\ell}_{\pi_\ell(e)}.
\]
The maximum-overlap statistic is
\[
        T \coloneqq \max_{\pi\in\Sn^{m-1}} T_\pi,
\]
where the maximum is over $(\pi_2,\cdots,\pi_m)$, with $\pi_1=\Id$. The maximum-overlap test rejects the null hypothesis when \(T\ge \tau\), where the threshold \(\tau\) is specified below.

Let
\(
    N \coloneqq \binom n2
\)
and 
\(
M \coloneqq \binom m2 .
\)
Define
\[
    \psi(p)\coloneqq \log\frac1p-1+p .
\]
Since \(p=n^{-o(1)}\), we have that $\log(1/p) = o(\log n)$, and so it follows from the assumption of the theorem that $N ps^2\to\infty$. Choose a sequence
$\delta_n\to 0$ such that
\(
    \delta_n^2 \cdot  N  ps^2\to\infty \, , 
\)
and set the threshold
\(
    \tau \coloneqq (1-\delta_n) M  Nps^2 .
\)

\paragraph{Controlling the type-II error} Under $ \P $, the score at the true alignment \(\pistar\) can be computed as follows.  For each vertex pair $e\in\En$,
\[
    \sum_{1 \leq k<\ell \leq m}
    X^k_{\pistar_k(e)} \cdot X^\ell_{\pistar_\ell(e)}
    =
    \sum_{1\leq k < \ell \leq m}
    \big( X^0_e  \xi^k_e \big) \cdot \big( X^0_e  \xi^\ell_e \big)
    \, =  \,
    X^0_e \binom{R_e}{2},
\]
where
\(
    R_e\coloneqq \sum_{k=1}^m \xi^k_e \sim \Bin(m,s).
\)
Thus
\[
    T_{\pistar}
    =
    \sum_{e\in\En} X^0_e \binom{R_e}{2}, 
\]
where the summands are independent, bounded by \(M \), and have mean
\(
    \Expect_\P\big[X^0_e\binom{R_e}{2}\big]
    =
    p M  s^2.
\)
Therefore, the mean score after summing over all the edges is
\(
        \Expect_\P[T_{\pistar}]=M Nps^2.
\)
Since \(T\ge T_{\pistar}\), we apply Bernstein's inequality. Note that
\begin{align*}
    \Var\left[ X_e^0 \binom {R_e} 2 \right]  
    &\leq  \Expect\left[ X_e^0  \binom {R_e} 2^2 \right] 
    = p \sum_{r=2}^m \binom m r s^r (1-s)^{m-r} \binom r 2 ^2 \\
    & \leq ps^2 \sum_{r=2}^m \binom m r s^{r-2} (1-s)^{m-r} \binom r 2 ^2 \\
    & \leq  ps^2 \sum_{r=2}^m \binom m r  \binom r 2 ^2 \leq ps^2 M^2 2^m.
\end{align*}
Therefore by Bernstein's inequality, for a constant $c_m$ depending only on $m$,
\begin{align*}
  \P(T<\tau)  & \leq  \exp\left(- \frac {\frac 1 2 (NMps^2\delta_n)^2} 
  {Nps^2 M^22^m + \frac 1 3 \binom m 2 MNps^2\delta_n }\right)  \\
  &=  \exp\left(- \frac {\frac 1 2 Nps^2\delta_n^2 M^2 } 
  { M^22^m + \frac 1 3 \binom m 2 M\delta_n }\right) \\
  &\leq  \exp\left(- Nps^2\delta_n^2 c_m\right) = o(1).
\end{align*} 

\paragraph{Controlling the type-I error}  Fix an arbitrary alignment
\(\pi\).  Under \(\Q\), the vectors
\(
    \big(X^1_{\pi_1(e)}, \cdots, X^m_{\pi_m(e)}\big)
\)
for
\(
    e\in\En,
\)
are i.i.d., with independent \(\Bern(ps)\) coordinates.  Hence
\[
    T_\pi
    \stackrel{ \mathrm{d}. }{=}
    \sum_{e\in\En} \binom{Z_e}{2},
    \qquad
    Z_e\stackrel{\mathrm{iid}}{\sim}\Bin(m,ps).
\]
Let \(Z\sim\Bin(m,ps)\), and set
\(
    \theta=\log(1/p).
\)
Chernoff's bound gives
\begin{align} \label{eq:chernoff-Q}
    \Q(T_\pi\ge \tau)
    \le
    \exp\left\{
        -\theta\tau
        +
        N\log \Expect \exp\left(\theta\binom Z2\right)
    \right\}.
\end{align}
Now
\begin{align} \label{eq:mgf-form}
    \Expect \exp\left(\theta\binom Z2\right)
    &=
    1+
    \sum_{r=2}^m
    \binom mr (ps)^r(1-ps)^{m-r}
    \left(p^{-\binom r2}-1\right).
\end{align}
Using \(1-ps \leq 1\), the \(r=2\) term contributes at most
\(
    M  (ps)^2(p^{-1}-1)
    =
    M ps^2(1-p)
\)
to the sum in~\eqref{eq:mgf-form}, while the contribution of all \(r\ge3\) terms is bounded by
\[
    O\left(
        \sum_{r=3}^m
        (ps)^r p^{-\binom r2}
    \right)
    =
    O\left(
        ps^2
        \sum_{r=3}^m
        s^{r-2} \, p^{-\binom{r-1}{2}}
    \right)
    =
    o(ps^2) \, .
\]
Therefore, using that $\log(1+x) \leq x$, we have
\begin{align} \label{eq:bound-on-mgf}
    \log \Expect \exp\left(\theta\binom Z2\right)
    \leq
    M ps^2(1-p+o(1)) \, .
\end{align}
It follows by plugging~\eqref{eq:bound-on-mgf} in~\eqref{eq:chernoff-Q}, along with the definition of $\theta$ and $\tau$, that
\begin{align*}
    -\log \Q(T_\pi\ge \tau)
    &\ge
    (1-\delta_n)M Nps^2\log\frac1p
    -
    M Nps^2(1-p+o(1))    \\
    &=
    M Nps^2
    \Big[
        \psi(p)-\delta_n\log\frac1p-o(1)
    \Big].
\end{align*}
Since $p \leq 1-c$, there is a constant $C'$ depending only on $c$ such that $ \log(1/p) \leq C' \psi(p)$. Since $\delta_n\to 0$, this gives
\[
    -\log \Q(T_\pi\ge \tau)
    \ge
    (1-o(1))M Nps^2\psi(p) \, .
\]
By the assumed lower bound on \((n-1)ps^2\) from the theorem statement, we have
\[
\begin{aligned}
        M Nps^2\psi(p)
        \, = \,
        \binom m2 \binom n2 ps^2\psi(p)    
        \, = \,
        \frac{m(m-1)}{4}\, n(n-1)ps^2\psi(p)  
        \, \ge \,
        (1+\epsilon)(m-1)n\log n .
\end{aligned}
\]
Thus, for all sufficiently large \(n\), we have
\(
    \Q(T_\pi\ge \tau)
    \le
    \exp\left\{
        -\left(1+\frac{\epsilon}{2}\right)(m-1)n\log n
    \right\},
\)
uniformly over \(\pi\).
Finally, taking a union bound over at most \((n!)^{m-1}\) alignments,
\[
    \begin{aligned}
        \Q(T\ge\tau)
        &\le
        (n!)^{m-1}
        \exp\left\{
            -\left(1+\frac{\epsilon}{2}\right)(m-1)n\log n
        \right\} \\[0.4em]
        &\leq
        \exp\left\{
            (m-1)n\log n
            -
            \left(1+\frac{\epsilon}{2}\right)(m-1)n\log n
        \right\}
        \, = \, 
        o(1) \, .
    \end{aligned}
\]
Combining the two bounds yields
\(
    \P(T<\tau)+\Q(T\ge\tau)=o(1) \, ,
\)
which completes the proof.
\end{document}